\documentclass[11pt]{article}

\usepackage{graphicx}
\usepackage{amsmath}
\usepackage{amsfonts}
\usepackage{color}
\usepackage{latexsym}
\usepackage{tikz}
\usepackage{pgfplots}
\pgfplotsset{compat=1.18} 
\usepackage{tablefootnote}
\usepackage{setspace}
\usepackage{epsfig}
\usepackage{multirow}
\usepackage{float}
\usepackage{array}
\usepackage{xcolor}
\usepackage{amssymb,amsthm}
\usepackage{xurl}
\usepackage{hyperref}    
\usepackage{algorithm2e} 
\usepackage{cleveref}    
\newcommand{\smallblacksquare}{\scalebox{0.6}{$\blacksquare$}}

\newcommand*{\R}{\mathbb{R}}

\newcommand{\ascal}[2]{\left\langle #1, #2 \right\rangle}
\renewcommand{\aa}[1]{\left\{ #1 \right\}}

\definecolor{antiquefuchsia}{rgb}{0.57, 0.36, 0.51}
\definecolor{MyViolet}{rgb}{0.45,0.08,0.95}
\definecolor{MyBrown}{rgb}{0.45,0.08,0}
\definecolor{MyDarkBlue}{rgb}{0,0.08,0.45}

\usepackage{makecell,booktabs}
\usepackage[version=4]{mhchem}   
\usepackage[strict]{changepage}

\usepackage{amsmath}
\usepackage{amsfonts}
\usepackage{latexsym}
\usepackage{amssymb}

\newtheorem{thm}{Theorem}[section]
\newtheorem{theorem}[thm]{Theorem}

\newtheorem{lemma}[thm]{Lemma}
\newtheorem{prop}[thm]{Proposition}

\newtheorem{definition}[thm]{Definition}

\newtheorem{rem}[thm]{Remark}
\newtheorem{remark}[thm]{Remark}

\newcommand{\beq}{\begin{equation}}
\newcommand{\eeq}{\end{equation}}
\newcommand{\beqa}{\begin{eqnarray}}
\newcommand{\eeqa}{\end{eqnarray}}
\newcommand{\beqas}{\begin{eqnarray*}}
\newcommand{\eeqas}{\end{eqnarray*}}
\newcommand{\bi}{\begin{itemize}}
\newcommand{\ei}{\end{itemize}}

\newcommand{\nn}{\nonumber}

\newcommand{\dom}{\mathrm{dom}\,}

\newcommand{\argmin}{\mathrm{argmin}\,}

\newcommand{\bConv}[1]{\overline{\mbox{\rm Conv}}\,(\R^{#1})}

\newcommand{\mConv}[1]{\overline{\mbox{\rm Conv}}_\mu\,(\R^{#1})}

\begin{document}
	\title{Stochastic Quadratic Dynamic Programming}
\date{}
 
	 \maketitle

\begin{center}
\begin{tabular}{ccc}
\begin{tabular}{c}
Vincent Guigues\\
School of Applied Mathematics, FGV\\
Praia de Botafogo, Rio de Janeiro, Brazil\\
{\tt vincent.guigues@fgv.br}
\end{tabular}&
&
\begin{tabular}{c}
Adriana Washington\\
School of Applied Mathematics, FGV\\
Praia de Botafogo, Rio de Janeiro, Brazil\\
{\tt adriana.washington07@gmail.com}
\end{tabular}
\end{tabular}
\end{center}

	\begin{abstract}
We introduce an algorithm called SQDP (Stochastic Quadratic Dynamic Programming) to solve some multistage stochastic optimization problems having strongly convex recourse functions.
The algorithm extends the classical Stochastic Dual Dynamic Programming
(SDDP) method replacing affine cuts by quadratic cuts.
We provide conditions ensuring strong convexity
of the recourse functions and prove the convergence of SQDP. 
In the special case of a single stage deterministic problem, we call
QCSC (Quadratic Cuts for Strongly Convex optimization) the method and prove its complexity.
Numerical experiments
illustrate the performance and correctness of
SQDP, with SQDP being much quicker than SDDP for large values of the constants
of strong convexity both for a multistage problem and a two-stage assembly recourse model. We also present the results of numerical experiments on deterministic problems where QCSC is much quicker than several
popular competing optimizers for solving 
6 strongly convex optimization problems from the literature.\\
\end{abstract}

		
	 
	\section{Introduction}\label{sec:intro}

In this paper, we introduce 
the Stochastic Quadratic Dynamic Programming (SQDP) method to solve
Multistage Stochastic Programs (MSPs) having strongly convex recourse functions.
To the best of our knowledge this method has not been proposed so far. This positions our research in the area of solution methods for MSPs. 
The most popular solution method 
for MSPs is the Stochastic Dual Dynamic Programming (SDDP) method, introduced in Pereira and Pinto (1991). SDDP is a sampling-based extension of the 
Nested Decomposition algorithm Birge (1985) which builds policies for some multistage stochastic convex problems.
It has been used to solve many real-life problems and several extensions of the method have been considered 
such as DOASA Philpott and Guan (2008), CUPPS Chen and Powell (1999), ReSA Hindsberger and Philpott (2001), AND Birge and Donohue (2001), risk-averse variants (Guigues and R\"omisch (2012a), Guigues and R\"omisch (2012b), Philpott and de Matos (2012), Guigues (2014), Shapiro (2011), Shapiro et al. (2009), Kozmik and Morton (2015)), inexact variants Guigues et al. (2021),
variants for problems with a random
number of stages Guigues (2021), or statistical
upper bounds Guigues et al. (2023), see also references therein and Shapiro (2011).
SDDP builds approximations for the cost-to-go functions which take the form of a maximum of
affine functions called cuts.

We propose an extension of SDDP method called SQDP (Stochastic Quadratic Dynamic Programming), which is a  
 Decomposition Algorithm for multistage stochastic programs having strongly convex cost functions.
Similarly to SDDP, at each iteration the algorithm computes in a forward pass a sequence of trial points which are used
in a backward pass to build lower bounding functions called cuts. However, contrary to SDDP where cuts are affine functions,
the cuts computed with SQDP   are quadratic functions
and therefore
the cost-to-go functions are approximated by a maximum of quadratic functions.
We show that for the MSPs considered the recourse functions are strongly convex
and that the quadratic cuts computed by SQDP are valid cuts, i.e., are lower
bounding functions for the corresponding cost-to-go functions.
We also show the convergence of SQDP and present numerical simulations showing the correctness of SQDP and comparing the performance of SQDP and SDDP on a simple MSP. On our experiments, SQDP is much quicker than SDDP for large values of the constants
of strong convexity.
For two-stage stochastic linear programs with continuous distributions, \cite{SCHULTZ19943} provides conditions ensuring that the recourse function is strongly convex. To solve such problems, we can therefore extend naturally SQDP
incorporating in Stochastic Decomposition \cite{stochasti_decomposition_higle_sen_91} 
quadratic cuts instead of affine cuts to approximate
the recourse function.
We consider in Section \ref{two_stage_example}
the application of this natural extension of SQDP for a two-stage
problem which is a multiproduct assembly model.

To motivate our method, we start our exposition with the case of 
deterministic strongly convex optimization problems and derive a method also based
on the computation of lower bounding quadratic cuts for the objective function. More precisely, given $\tilde f: \R^{n} \rightarrow \R\cup \{ +\infty \}$, a  proper lower semi-continuous $\mu$-convex function, we consider the problem
 \begin{equation}
\label{pb:initial}
 \tilde f_* =   \min \left\{\tilde f(x) \;|\; x \in X\right\}
\end{equation}
for a compact convex set $X$. For this problem, we consider an extension of Kelley's cutting plane method
Kelley (1960) which consists in replacing the affine cuts in Kelley's algorithm by valid lower bounding quadratic cuts.
We call QCSC (Quadratic Cuts for Strongly Convex problems) the corresponding method, which, to the best of our knowledge, is new.
We also prove the complexity of QCSC.

\par The outline of the paper is as follows. In Section \ref{motivating}, we propose a simple algorithm
for minimizing a strongly convex function which builds a model
of the objective on the basis of quadratic lower approximations
of this objective computed at trial points. In this section, we also propose
a reformulation of QCSC in terms of affine cuts plus a quadratic term
independent on the trial points. 
This reformulation allows us to
adapt the complexity from 
Liang and Monteiro (2023) to obtain the complexity of QCSC.
In Section \ref{sec:num_qcsc}
we show  that QCSC 
converges much quicker than several competing methods (limited memory Proximal Bundle Method \cite{kiw90,ds16}; the Level Bundle Method  \cite{LemarechalNewVariantsBundle1995_lvl}; the universal proximal bundle method \cite{guigues2024universal_upb}; the Subgradient with Restart method  \cite{RenegarGrimmer2022_subgrad_rs}) for 6 strongly convex minimization problems. In Section \ref{sec:pbformass},
we define the class of MSPs with strongly convex cost functions we are interested to solve. In particular, in this section, we show the strong convexity of the corresponding recourse functions.
In Section \ref{sec:SQDP}, we present our
SQDP algorithm. Section 
\ref{convanalysis} is dedicated to the
convergence analysis of SQDP.
Numerical experiments comparing
SQDP with SDDP on a MSP are given in
Section \ref{sec:num}.
In this section, we also
consider a real-life two-stage application
and solve this problem using Stochastic Decomposition that uses affine cuts as well as the adaptation of SQDP to this setting using quadratic cuts.

Concluding remarks and possible extensions are discussed in Section \ref{sec:conc}.

    \subsection{Basic definitions and notation.} Let $\R$ denote the set of real numbers.
    Let $ \R_+ $ and $ \R_{++} $ denote the set of non-negative real numbers and the set of positive real numbers, respectively.
	Let $\R^n$ denote the standard $n$-dimensional Euclidean space equipped with  inner product and norm denoted by $\left\langle \cdot,\cdot\right\rangle $
	and $\|\cdot\|$, respectively. 
	Let $\log(\cdot)$ denote the natural logarithm.

	Let $\Psi: \R^n\rightarrow (-\infty,+\infty]$ be given. Let $\dom \Psi:=\{x \in \R^n: \Psi (x) <\infty\}$ denote the effective domain of $\Psi$.
 We say that
	$\Psi$ is proper if $\dom \Psi \ne \emptyset$.
	A proper function $\Psi: \R^n\rightarrow (-\infty,+\infty]$ is $\mu$-convex for some $\mu \ge 0$ if
	$$
	\Psi(\alpha z+(1-\alpha) z')\leq \alpha \Psi(z)+(1-\alpha)\Psi(z') - \frac{\alpha(1-\alpha) \mu}{2}\|z-z'\|^2
	$$
	for every $z, z' \in \dom \Psi$ and $\alpha \in [0,1]$.
 The set of all proper lower semicontinuous $\mu$-convex functions is denoted by $\mConv{n}$.
 When $\mu=0$, we simply denote
    $\mConv{n}$ by $\bConv{n}$.
The subdifferential of $ \Psi $ at $z \in \dom \Psi$ is denoted by
	\begin{equation}
        \partial \Psi (z):=\left\{ s \in\R^n: \Psi(z')\geq \Psi(z)+\left\langle s,z'-z\right\rangle, \forall z'\in\R^n\right\}.
        \end{equation}
For a set 
$\mathcal{X}$, its $\varepsilon$-enlargement $\mathcal{X}^{\varepsilon}$
is given by
$\mathcal{X}^{\varepsilon}=\{x: \mbox{dist}(x,\mathcal{X}) \leq \varepsilon\}$.

\section{QCSC algorithm}\label{motivating}

\subsection{Statement of the algorithm}

We consider  optimization problem 
\eqref{pb:initial}
where $\tilde f$ is a convex lower-semicontinuous function, subdifferentiable on $X$ a nonempty compact convex subset of $\R^n$.
The Kelley cutting plane method applied to problem~\eqref{pb:initial} consists in approximating the convex function $\tilde f$ 
by a polyhedral approximation that we iteratively refine. 
More precisely, let $\tilde f'$ define a selection of $\partial \tilde f$, i.e., a function such that, for all $x\in X$, $\tilde f'(x) \in \partial \tilde f(x)$.
Further,  denote $\ell_{\tilde f}(\cdot,x):= \tilde f(x) + \ascal{\tilde f'(x)}{\cdot - x}$ , for any $x \in X$, the tangent of $\tilde f$ at $x$ of slope $\tilde f'(x)$.
In particular, as $\tilde f$ is convex we have $\ell_{\tilde f}(\cdot,x) \leq \tilde f$, and any trial point $x_k \in X$ defines a so-called (affine) cut $\ell_{\tilde f}(\cdot,x_k) \leq \tilde f$.
Then, the Kelley algorithm, given a collection of trial points $(x_j)_{j=0,\ldots,k-1}$ at iteration $k$,
defines  a model $\Gamma_k$ of $\tilde f$ as the maximum of past cuts, i.e., $ \max_{j=0,\ldots,k-1} \ell_{\tilde f}(\cdot,x_j) \leq \tilde f$. 
The next trial point $x_{k}$ is defined as a minimum of $\Gamma_k$ over $X$. 
This is detailed in~
Algorithm \ref{alg:Kelley}, where we can note that, at each iteration, we have a candidate solution $y_k  \in X$ (the best solution found so far among the trial points), and an optimality gap $t_k \geq 0$. 

\begin{algorithm}
\par {\textbf{Inputs:}} $x_0 \in X$, $\bar \varepsilon>0$.\\
$\Gamma_1(\cdot) = \ell_{\tilde f}(\cdot,x_0) $;\\
$t_0=+\infty$;\\
$k=1$;
\par {\textbf{while }}$t_{k-1}>\bar \varepsilon$,\\
    $\hspace*{0.7cm}x_{k} \in \argmin_{x \in X} \Gamma_k(x) $ ;\\
    $\hspace*{0.7cm}y_k \in \argmin \{{\tilde f}(x): x \in \{x_0,x_1,\ldots,x_k\}\}$;\\
    $\hspace*{0.7cm}t_{k} = {\tilde f}(y_{k}) - \Gamma_k(x_{k})$;\\
   $\hspace*{0.7cm}\Gamma_{k+1}(\cdot) \leftarrow \max \aa{\Gamma_{k}(\cdot),\ell_{\tilde f}(\cdot,x_{k})}$;\\
   $\hspace*{0.7cm}k \leftarrow k+1$;
\par {\textbf{end}}
\par {\textbf{Outputs:}} approximate solution $y_{k-1}$, gap $t_{k-1}$.\\
\vspace*{0.3cm}
\caption{Kelley's cutting plane algorithm}\label{alg:Kelley}
\end{algorithm}

This algorithm is well-known and studied but only leverages the convexity of $\tilde{f}$. 
Assume now that $\tilde f$ is $\mu$-convex on $X$.
Then we can replace in Kelley's cutting plane  algorithm
the affine cuts $\ell_{\tilde f}(\cdot,x_k)$ by the
quadratic cuts 
\begin{equation}
\label{eq:qf}
    q_{\tilde f}(\cdot,x_k)= \underbrace{{\tilde f}(x_k) + \ascal{{\tilde f}'(x_k)}{\cdot - x_k}}_{\ell_{\tilde f}(\cdot,x_k)} + \frac{\mu}{2}\| \cdot - x_k\|^2 \leq \tilde f(\cdot),
\end{equation}
yielding a new, to our knowledge, algorithm that we call QCSC (Quadratic Cuts for Strongly Convex optimization).
Notice that inequality \eqref{eq:qf} 
directly follows from
the strong convexity of
$\tilde f$ and shows that
quadratic approximations
$q_{\tilde f}$ of $\tilde f$ are better
than affine approximations
$\ell_{\tilde f}$.

QCSC is given in Algorithm
\ref{alg:qcsc}.

\begin{algorithm}
\par {\textbf{Inputs:}} $x_0 \in X$, $\bar \varepsilon>0$.\\
$\Gamma_1(\cdot) =q_{\tilde f}(\cdot,x_0)$;\\
$t_0=+\infty$;\\
$k=1$;
\par {\textbf{while }}$t_{k-1}>\bar \varepsilon$,\\
    $\hspace*{0.7cm}x_{k} \in \argmin_{x \in X} \Gamma_k(x) $ ;\\
    $\hspace*{0.7cm}y_k \in \argmin \{{\tilde f}(x): x \in \{x_0,x_1,\ldots,x_k\}\}$;\\
    $\hspace*{0.7cm}t_{k} = {\tilde f}(y_{k}) - \Gamma_k(x_{k})$;\\
   $\hspace*{0.7cm}\Gamma_{k+1}(\cdot) \leftarrow \max \aa{\Gamma_{k}(\cdot),q_{\tilde f}(\cdot,x_k)}$;\\
   $\hspace*{0.7cm}k \leftarrow k+1$;
\par {\textbf{end}}
\par {\textbf{Outputs:}} approximate solution $y_{k-1}$, gap $t_{k-1}$.\\
\vspace*{0.3cm}
\caption{QCSC algorithm}\label{alg:qcsc}
\end{algorithm}

As an illustration, we implemented QCSC with
$X=[-10,10]$ and
one-dimensional objective function 
$$
\tilde f(x)=\max(1000(x-4)^2+2,1000(x+5)^2+8,500(x-3)^2+6),
$$ 
which is nondifferentiable and strongly convex.
For starting point $x_0=8$, the approximate optimal value stabilized
after 4 iterations and the objective function together with the first two quadratic approximations
$q_{{\tilde f}}(\cdot,x_0)$
and 
$q_{{\tilde f}}(\cdot,x_1)$
computed by QCSC (the corresponding model being the maximum of these two quadratic cuts) are represented in Figure \ref{fig:QCC}.

\begin{figure}
    \centering
\includegraphics[width=0.75\textwidth]{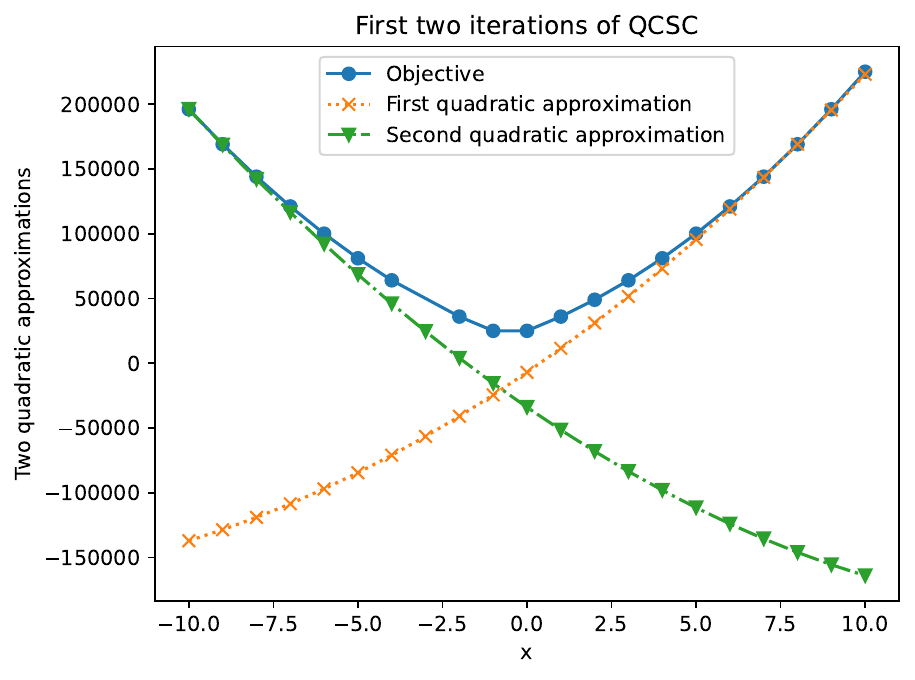}
    \caption{Two iterations of QCSC. The first quadratic approximation is given in orange while the second quadratic approximation is given in green. The model at the second iteration is the maximum of the two quadratic cuts. The objective function which is approximated is the strongly convex function in blue.}
    \label{fig:QCC}
\end{figure}

\subsection{A reformulation of QCSC}

In this section, we provide a reformulation of
 QCSC in terms of affine (instead of quadratic in the
 original QCSC formulation) cuts plus a quadratic
 term $(\mu/2)\|\cdot\|^2$ independent on the trial points.
We first write the original problem \eqref{pb:initial} under the form
\begin{equation}\label{reform0}
\min_{x \in \mathbb{R}^n} \tilde f(x) + \tilde h(x)
\end{equation}
where 
$\tilde h(x)=\delta_X(x)$ is the indicator function of set $X$.
We then define functions $f$ and
$h$ by
\begin{equation}\label{defftildehtilde}
f(x)=\tilde f(x)-\frac{\mu}{2}\|x\|^2 \mbox{ and }h(x)=\tilde h(x)+\frac{\mu}{2}\|x\|^2.
\end{equation}
Observe that problem \eqref{pb:initial}
can be written
\begin{equation}\label{reformcomposite}
\min_{x \in \mathbb{R}^n} \;f(x) + h(x).
\end{equation}
Since $\tilde h$ is convex we have that
$h$
is $\mu$-strongly convex. We also have that
$\tilde f \in \mConv{n}$
is $\mu$-convex
and therefore
$f$ is convex.

Note that
there is $\tilde f'(x) \in \partial \tilde f(x)$
and 
$f'(x) \in \partial f(x)$
such that 
$f'(x)=\tilde f'(x)-\mu x$ and we can write
\begin{equation}\label{reformlincut}
\begin{array}{lcl}
\ell_f(u,x)+h(u) & = & f(x) + \langle f'(x),u-x\rangle +h(u) \\
& = & f(x) + \langle {\tilde f}'(x)-\mu x,u-x\rangle +\tilde h(u) +\frac{\mu}{2}\|u\|^2\\
& = & \tilde f(x)-\frac{\mu}{2}\|x\|^2 + \langle {\tilde f}'(x),u-x\rangle - \mu\langle x, u-x \rangle +\tilde h(u) +\frac{\mu}{2}\|u\|^2\\
&=& q_{\tilde f}(u,x)+\tilde h(u),
\end{array}
\end{equation}
which shows that the model written in terms
of quadratic cuts $q_{\tilde f}(u,x)$
can be written in terms
of affine cuts 
$\ell_f(u,x)$ 
written for $f$ instead
of $\tilde f$ but also
replacing $\tilde h$ by $h$, which is now $\mu$-convex, i.e., adding to the affine cuts a quadratic term $(\mu/2)\|\cdot\|^2$ independent on the trial
points.
We see that we have transferred
the $\mu$-convexity from $\tilde f$ to
$h$ 
and the convexity of $\tilde h$
to $f$ for which affine cuts
can be computed.
Recall that 
the objective function of composite
optimization problem \eqref{reform0} (which is equivalent to the original problem \eqref{pb:initial}) is the sum
of a convex function $\tilde f$ which is
approximated in QCSC  by a model 
constructed from quadratic cuts
$q_{\tilde f}(\cdot,x)$ of $\tilde f$
and of convex function $\tilde h$.
This problem can also be expressed
as composite optimization problem
\eqref{reformcomposite}
with objective function which is
the sum of a $\mu$-convex
function $h$ and
convex function $f$. We obtain the reformulation of
QCSC given in Algorithm \ref{alg:qcsc2} below.

\begin{algorithm}
\par {\textbf{Inputs:}} $x_0 \in X$, $\bar \varepsilon>0$.\\
$\Gamma_1(\cdot) =\ell_{f}(\cdot,x_0) +\frac{\mu}{2}\|\cdot\|^2$;\\
$t_0=+\infty$;\\
$k=1$;
\par {\textbf{while }}$t_{k-1}>\bar \varepsilon$,\\
    $\hspace*{0.7cm}x_{k} \in \argmin_{x \in X} \Gamma_k(x) $ ;\\
    $\hspace*{0.7cm}y_k \in \argmin \{{\tilde f}(x): x \in \{x_0,x_1,\ldots,x_k\}\}$;\\
    $\hspace*{0.7cm}t_{k} = {\tilde f}(y_{k}) - \Gamma_k(x_{k})$;\\
   $\hspace*{0.7cm}\Gamma_{k+1}(\cdot) \leftarrow \max \aa{\Gamma_{k}(\cdot),\ell_{f}(\cdot,x_k) +\frac{\mu}{2}\|\cdot\|^2}$;\\
   $\hspace*{0.7cm}k \leftarrow k+1$;
\par {\textbf{end}}
\par {\textbf{Outputs:}} approximate solution $y_{k-1}$, gap $t_{k-1}$.\\
\vspace*{0.3cm}
\caption{Reformulation of QCSC algorithm}\label{alg:qcsc2}
\end{algorithm}

\subsection{Complexity analysis of QCSC}

Our complexity analysis of QCSC 
is based on the reformulation of
QCSC given in Algorithm \ref{alg:qcsc2}
that uses function $f$ given by
\eqref{defftildehtilde}.
Our analysis closely follows the complexity of the bundle
method from Liang and Monteiro (2023). However, some adaptations
are necessary due to the fact that there is an additional prox-term
in Liang and Monteiro (2023) which is not present in QCSC (since cuts
are already quadratic in QCSC) and we do not use any black-box. Even without prox-term
and due to the presence of quadratic cuts, the complexity we obtain
for QCSC is similar to the complexity of the bundle methods from
Liang and Monteiro (2023). 

We assume that the following conditions
hold
	for some triple
	$(L_f, M_f,\mu) \in \R_+^3 $:
	\begin{itemize}
		\item[(A1)]
		$\tilde f$ is $\mu$-convex for some $\mu>0$, 
		$X \subset \dom \tilde f$, and a subgradient oracle, i.e.,
		a function $\tilde f':X \to \R^n$
		satisfying $\tilde f'(x) \in \partial \tilde f(x)$ for every $x \in X$, is available;
		\item[(A2)]
		the set of optimal solutions $X^*$ of
		problem \eqref{pb:initial} is nonempty;
		\item[(A3)]
		for every $u,v \in \dom h$,
		\begin{equation}
		\|f'(u)-f'(v)\| \le 2M + L \|u-v\|,
		\end{equation}
        where $f'(x)=\tilde f'(x)-\mu x$.
	\end{itemize}

    It is well-known that (A3) implies that for every $u,v \in \dom h$,
	\begin{equation}\label{ineq:est}
	f(u)-\ell_f(u;v) \le 2M \|u-v\| + \frac{L}{2}\|u-v\|^2.
	\end{equation}
    
Indeed, relation \eqref{ineq:est} comes from the relations
\begin{align*}
f(u)-\ell_f(u;v) &=
\displaystyle \int_{0}^1
\left \langle u-v,f'(v+t(u-v))-f'(v) \right \rangle dt\\
& \leq \displaystyle \int_{0}^1 \|u-v\|\Big(
2M + L\|u-v\|t \Big)dt\\
&=2M \|u-v\| + \frac{L}{2}\|u-v\|^2,
\end{align*}
where the  inequality comes from
(A3) and
Cauchy-Schwarz inequality.

In addition to the above quantities, we introduce
the
	distance of initial
	point $x_0$ to $X^*$:
\begin{equation}\label{def:d0}
	d_0 := \|x_0-x_0^*\| \ \ \mbox{\rm where}  \ \ \ 
	x_0^* := \argmin \{\|x_0-x^*\|: x^*\in X^*\}.
\end{equation}

		 Finally,
for given initial point $x_0 \in X$
and tolerance $\bar \varepsilon>0$, it is said that an algorithm for solving \eqref{pb:initial}
has $\bar \varepsilon$-iteration complexity ${\cal O}(N)$ if its total number of iterations until it obtains a $ \bar \varepsilon$-optimal solution
 is bounded by 
 $C(N+1)$ where $C>0$ is a
universal constant. 

The next theorem gives the
 complexity of QCSC.
 
    	\begin{theorem}[Complexity of QCSC]\label{complcssc} 
    Assume that $X$
    is bounded with diameter $D>0$ and let $\bar \varepsilon>0$. 
    Then if Assumptions (A1), (A2), and (A3) hold and if
\begin{equation}\label{condj}
k \geq 1 + \left[1+ \frac{8\left( M^2+\bar \varepsilon L \right)}{\mu \bar \varepsilon}
\right]\log \left( \frac{4 \bar t(D)}{3 \bar \varepsilon}  \right)
\end{equation}
where 
\begin{equation}\label{def:bar t}
		    \bar t(D):= M^2 + 
       \left( \frac{{L}}{2} +1\right) D^2.
		\end{equation}
we have $t_k \leq \bar \varepsilon$ and
QCSC finds an $\bar \varepsilon$-optimal solution
of \eqref{pb:initial} in at most 
$$
1 + \left[1+ \frac{8\left( M^2+\bar \varepsilon L \right)  }{\mu \bar \varepsilon}
\right]\log \left( \frac{4 \bar t(D)}{3 \bar \varepsilon}  \right)
$$
iterations.
	\end{theorem}

The
proof of Theorem
\ref{complcssc}
is given in Appendix~\ref{appendix_qcsc_complexity}.

\section{Numerical Experiments for QCSC}\label{sec:num_qcsc}
In this section, we report numerical experiments illustrating the practical performance of QCSC comparing 
this method 
with methods from the literature on a collection of test problems. For all methods, the stopping criterion is based on the computation of an $\varepsilon$-solution, and we report both CPU time and the number of iterations required to reach it.

All algorithms were implemented in Julia and executed on a single thread. When needed, MOSEK was used to solve the convex optimization subproblems arising within the methods. When the optimal value was required, it was computed using IPOPT. All numerical experiments in this section were conducted on a machine equipped with an Apple M1 processor and 8~GB of RAM.

We consider strongly convex reformulations of benchmark problems from the literature, obtained by adding a quadratic regularization term to their standard formulations. The methods included in our comparison are the limited memory Proximal Bundle Method, originally proposed in \cite{kiw90} and implemented following the recommendations in \cite{ds16}, the Level Bundle Method (LVL) \cite{LemarechalNewVariantsBundle1995_lvl}, the universal proximal bundle method (U-PB) \cite{guigues2024universal_upb}, the Subgradient with Restart method (Subgrad RS) \cite{RenegarGrimmer2022_subgrad_rs}, and the method proposed in this paper: QCSC. For methods requiring parameter tuning, we use the following settings: for LVL, $\lambda = 0.5$; for PBM-1, $\lambda_0 \in \{10.0, 1.0, 10^{-3}\}$, $\lambda_{\min} \in \{10^{-4}, 10^{-6}\}$, $a = 2.0$, and $k = 0.1$; and for U-PB, $\lambda_0 \in \{0.1, 1.0\}$, $\chi = 0.5$, and $\bar{N} \in \{5, 20\}$. Consider $\mu > 0$ the strong convexity parameter and $x_0 \in\mathbb{R}^n$ the initial point. The code for these experiments can be found at \url{https://github.com/AdrianaWH/Numerical-Experiments-QCSC-method}.

\subsection{Strongly Convex BadGuy}\label{sec:num_qcsc_badguy}

We consider a strongly convex variant of the BadGuy problem, a classical pathological example for cutting-plane methods introduced in \cite{Hiriart-UrrutyConvexAnalysisMinimization1993} and used in \cite{SagastizabalCompositeProximalBundle2013}. As in the original formulation, the feasible set is the unit ball of $\mathbb{R}^n\times\mathbb{R}$. More precisely, for $\varepsilon\in(0,1/2)$, we consider the minimization over $B=\left\{(y,\eta)\in\mathbb{R}^n\times\mathbb{R}:\|y\|_2^2+\eta^2\le 1\right\}$ of the function
\begin{equation}
f_\mu(y,\eta):=
\max\left\{|\eta|,\,-1+2\varepsilon+\|y\|_2\right\}
+\frac{\mu}{2}\left(\|y\|_2^2+\eta^2\right).
\end{equation}
Hence, the strongly convex term is added to the original BadGuy objective while keeping the same theoretical feasible set. The corresponding results  are reported in Table \ref{tab_badguy}. It can be seen that QCSC achieves the best performance
with the smallest computational time and converging in just two iterations.

\begin{table}[H]
\centering
\begin{tabular}{|c|c|c|}
\hline
\textbf{Method} & Time(s) & Iterations \\
\hline
LVL         & 4.2964 & 29 \\
\hline
PBM-1       & 0.4865 & 12 \\
\hline
QCSC & 0.0509 & 2 \\
\hline
SUBGRAD RS  & 0.1206 & 5277 \\
\hline
U-PB        & 0.4942 & 2  \\
\hline
\end{tabular}
\caption{Results for BadGuy Strongly Convex ($\varepsilon = 10^{-6},\; n = 10, \mu = 1000.0, x_0 = (1,\ldots,1)$).}
\label{tab_badguy}
\end{table}

\subsection{Strongly Convex CB3 II}\label{sec:num_qcsc_cb3ii}

We consider the reformulation of the Chained CB3 II problem, a classical nonsmooth optimization problem introduced in \cite{CharalambousNonlinearMinimaxOptimization1976} and reported in \cite{HaaralaNewLimitedMemory2004}. More precisely, we minimize
\begin{equation}
\tilde{f}_2(x) := \max \left\{ \sum_{i=1}^{n-1}(x_i^4 + x_{i+1}^2), \sum_{i=1}^{n-1} (2-x_i)^2 + (2-x_{i+1})^2, \sum_{i=1}^{n-1} 2\exp(-x_i + x_{i+1})  \right\} + \frac{\mu}{2}\|x\|^2.
\end{equation}
In the numerical experiments, we consider the initial point $x_0 = 0$. The corresponding results are reported in Table \ref{tab_cb3}, where it is possible to verify that  QCSC maintained superior performance across all tested dimensions, with the smallest CPU times and a reduced number of iterations, in addition to demonstrating good scalability with increasing $n$.

\begin{table}[H]
\centering
\begin{tabular}{|c|c|c|c|c|c|c|}
\hline
 & \multicolumn{2}{c|}{$n = 500$} 
 & \multicolumn{2}{c|}{$n = 1000$} 
 & \multicolumn{2}{c|}{$n = 2000$} \\
\hline
\textbf{Method} 
& Time(s) & Iterations 
& Time(s) & Iterations 
& Time(s) & Iterations \\
\hline
LVL         & 6.5599 & 41 & 6.4692 & 62 & 7.2641 & 38 \\
\hline
PBM-1       & 0.4602 & 8 & 0.4549 & 8 & 0.5226 & 8 \\
\hline
QCSC & 0.0293 & 6 & 0.0542 & 7 & 0.1317 & 7 \\
\hline
SUBGRAD RS  & 0.118 & 715 & 0.1981 & 1429 & 0.9308 & 2859 \\
\hline
U-PB        & 0.5307 & 22 & 0.6603 & 25 & 0.892 & 26 \\
\hline
\end{tabular}
\caption{Results for CB3 II Strongly Convex ($\varepsilon = 10^{-1}, \mu = 10.0, x_0 = (0,\ldots,0)$).}
\label{tab_cb3}
\end{table}

\subsection{Strongly Convex MaxQuad}\label{sec:num_qcsc_maxquad}

We consider a reformulation of the well-known MaxQuad problem, commonly used to test methods for nondifferentiable convex optimization, introduced by Lemaréchal et al. in \cite{bonnansNumericalOptimizationTheoretical2006_maxquad}. More precisely, we study the problem, with $n=10$, $N=5$, of minimizing over $\mathbb{R}^n$:
\begin{equation}
\tilde f_3(x) = \max_{1 \leq \ell \leq N} \left( \langle A_{\ell} x, x \rangle + \langle b_{\ell}, x \rangle \right) + \frac{\mu}{2} \|x\|^2,
\end{equation}
 where the vectors $b_\ell$ and matrices $A_\ell$ are defined as in the standard 
 MaxQuad problem
 formulation.
\\
 The corresponding results for this strongly convex formulation are reported in Table \ref{tab_maxquad}. Once again, QCSC shows the best performance among the tested methods.
However, it is important to highlight the reduction in CPU time and number of iterations as the value of $\mu$ increases.

\begin{table}[H]
\centering
\begin{tabular}{|c|c|c|c|c|c|c|}
\hline
 & \multicolumn{2}{c|}{$\mu = 1.0$} 
 & \multicolumn{2}{c|}{$\mu = 10.0$} 
 & \multicolumn{2}{c|}{$\mu = 100.0$} \\
\hline
\textbf{Method} 
& Time(s) & Iterations 
& Time(s) & Iterations 
& Time(s) & Iterations \\
\hline
LVL         & 4.7475 & 142 & 5.4057 & 146 & 4.7954 & 133 \\
\hline
PBM-1       & 0.5828 & 95 & 0.5274 & 76 & 0.5721 & 65 \\
\hline
QCSC & 0.5043 & 194 & 0.0799 & 43 & 0.0595 & 15 \\
\hline
SUBGRAD RS  & 1.0778 & 26981 & 1.4678 & 33399 & 0.5083 & 12629 \\
\hline
U-PB        & 24.015 & 15528 & 20.0802 & 12713 & 8.5642 & 5231 \\
\hline
\end{tabular}
\caption{Results for Strongly Convex MaxQuad ($\varepsilon = 10^{-4}, n = 10, x_0 = (1,\ldots,1)$).}
\label{tab_maxquad}
\end{table}

\subsection{Strongly Convex MXHILB}\label{sec:num_qcsc_mxhilb}

We consider a reformulation of the MXHILB problem, associated with the computation of the kernel of the Hilbert matrix and introduced by Kiwiel \cite{KiwielEllipsoidTrustRegion1989}. We study the problem of minimizing over $\mathbb{R}^n$ the function
\begin{equation}
\tilde{f}_4(x) := \max_{1 \leq i \leq n} \left\{ |x_i| \sum_{j=1}^n \frac{1}{i+j-1} \right\} + \frac{\mu}{2}\|x\|^2.
\end{equation}
The corresponding results are presented in Figure \ref{result_mxhilb}, where it is possible to observe consistent behavior for QCSC, whose CPU times remain stable despite variations in $n$, with its advantage becoming more pronounced as $n$ increases. The results for PBM-1 are not included in the figure, as the method faces convergence problems for values of $n$ greater than 3000 given the fixed parameter configuration that was used for all methods to ensure a fair comparison. 
\begin{figure}[H] 
    \centering
\includegraphics[width=0.70\textwidth]{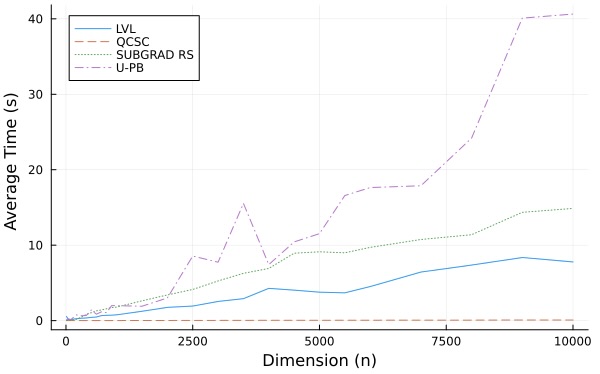}
    \caption{CPU time (in seconds) comparison for dimensions $n$ in \{10, 50, 100, 200, 300, 400, 500, 600, 700, 800, 900, 1000, 1500, 2000, 2500, 3000, 3500, 4000, 4500, 5000, 5500, 6000, 7000, 8000, 9000, 10000\}, considering $\mu=100.00$ and $x_0=(1,\ldots,1)\in\mathbb{R}^n$. Every CPU time is obtained
    averaging the CPU time on 20 instances of size $n$ for each value of $n$.}\label{result_mxhilb}
\label{mxhilb_dimension_average_time}
\end{figure}

\subsection{Strongly Convex RandMaxQuad}\label{sec:num_qcsc_randmaxquad}

We consider the reformulation of the box-constrained optimization problem, commonly used to test methods for nonsmooth convex optimization, introduced in \cite{ds16}. More precisely, we consider the optimization problem
\begin{equation}
\min_{\underline{x} \leq x \leq \overline{x}} \; \tilde{f}_5(x) := \max_{1 \leq i \leq N} \left( x^{\top} A_i x + b_i^{\top} x + \alpha \|x\|_1 \right) + \frac{\mu}{2}\|x\|^2.
\end{equation}
In the numerical experiments, we take $\alpha = 0.5$ and $N = 5$. The matrices $A_i$ are randomly generated symmetric positive definite  with condition number equal to $n$, and $b_i$ are random vectors. The corresponding results are reported in Table \ref{tab_randmaxquad}. Although the performance gains are not large in magnitude, QCSC achieves the best overall performance among the competing methods.

\begin{table}[H]
\centering
\begin{tabular}{|c|c|c|c|c|c|c|}
\hline
 & \multicolumn{2}{c|}{$\mu = 10$} 
 & \multicolumn{2}{c|}{$\mu = 100$} 
 & \multicolumn{2}{c|}{$\mu = 1000$} \\
\hline
\textbf{Method} 
& Time(s) & Iterations 
& Time(s) & Iterations 
& Time(s) & Iterations \\
\hline
LVL         & 7.2816  & 115 & 6.5719 & 117 & 8.2095 & 126 \\
\hline
PBM-1       & 0.8356 & 94 & 1.0898 & 94 & 0.9328 & 96 \\
\hline
QCSC & 0.1066 & 42 & 0.0852 & 22 & 0.0917 & 10 \\
\hline
SUBGRAD RS  & 0.1873 & 1641 & 0.1062 & 572 & 0.3487 & 3129  \\
\hline
U-PB        & 36.5684 & 15817 & 1.0176 & 296 & 1.0119 & 194 \\
\hline
\end{tabular}
\caption{Results for Strongly Convex RandMaxQuad ($\varepsilon = 10^{-4}, n = 10, x_0 = (1,\ldots,1)$).}
\label{tab_randmaxquad}
\end{table}

\subsection{Strongly Convex TiltedNorm}\label{sec:num_qcsc_tiltednorm}

We consider a reformulation of the TiltedNorm problem, a classical nonsmooth optimization problem used to test bundle-type methods (see \cite{skajaaLimitedMemoryBFGS, SagastizabalCompositeProximalBundle2013, ds16}). Considering $X = [-2,2]^n$, we study
\begin{equation}
\tilde{f}_6(x) := \alpha \|Ax\|_2 + \beta e_1^T A x + \frac{\mu}{2}\|x\|^2,
\end{equation}
where $\alpha \geq \beta$ are given parameters, $A$ is a symmetric positive definite matrix and $e_1 = (1,0,\ldots,0)$.
\\
The corresponding results are reported in Table \ref{tab_tilted}. We observe that QCSC achieved the smallest CPU times and number of iterations, while maintaining efficiency in both tested cases.

\begin{table}[H]
\centering
\begin{tabular}{|c|c|c|c|c|}
\hline
 & \multicolumn{2}{c|}{$n = 100;\; \mu=100.0$} 
 & \multicolumn{2}{c|}{$n = 200;\; \mu=1000.0$} \\
\hline
\textbf{Method} 
& Time(s) & Iterations 
& Time(s) & Iterations \\
\hline
LVL         & 11.2699 & 321 & 48.0846 & 479 \\
\hline
PBM-1       & 0.7889 & 186 & 1.7431 & 229 \\
\hline
QCSC & 0.3859 & 101 & 1.4814 & 125 \\
\hline
SUBGRAD RS  & 0.5585 & 7457 & 31.0635 & 85812 \\
\hline
U-PB        & 4.8807 & 2507 & 10.9661 & 2662 \\
\hline
\end{tabular}
\caption{Results for TiltedNorm ($\varepsilon = 10^{2}, \alpha = 4, \beta = 3, x_0 = (1,\ldots,1)$).}
\label{tab_tilted}
\end{table}

\section{Multistage stochastic programs with strongly convex cost functions}\label{sec:pbformass}

We consider multistage stochastic optimization problems (MSPs) of the form
\begin{equation}\label{pbtosolve}
\begin{array}{l} 
\displaystyle{\inf_{x_1,\ldots,x_T}} \; \mathbb{E}_{\xi_2,\ldots,\xi_T}[ \displaystyle{\sum_{t=1}^{T}}\;f_t(x_t(\xi_1,\xi_2,\ldots,\xi_t), x_{t-1}(\xi_1,\xi_2,\ldots,\xi_{t-1}), \xi_t )]\\
x_t(\xi_1,\xi_2,\ldots,\xi_t) \in X_t( x_{t-1}(\xi_1,\xi_2,\ldots,\xi_{t-1}), \xi_t )\;\mbox{a.s.}, \;x_{t} \;\mathcal{F}_t\mbox{-measurable, }t=1,\ldots,T,
\end{array}
\end{equation}
where $x_0$ is given, $\xi_1$ is deterministic,  $(\xi_t)_{t=2}^T$ is a stochastic process, $\mathcal{F}_t$ is the sigma-algebra
$\mathcal{F}_t:=\sigma(\xi_j, j\leq t)$, and $X_t(x_{t-1}, \xi_t ),t=1,\ldots,T$, can be of two types:
\begin{itemize}
\item[(S1)] $X_t( x_{t-1}, \xi_t ) = \{x_t \in \mathbb{R}^n : x_t \in \mathcal{X}_t : x_t \geq 0,\;\;\displaystyle A_{t} x_{t} + B_{t} x_{t-1} = b_t \}$ (in this case, for short, we say that $X_t$ is of type S1); 
\item[(S2)] $X_t( x_{t-1} , \xi_t)= \{x_t \in \mathbb{R}^n : x_t \in \mathcal{X}_t,\;g_t(x_t, x_{t-1}, \xi_t) \leq 0,\;\;\displaystyle A_{t} x_{t} + B_{t} x_{t-1} = b_t \}$.
In this case, for short, we say that $X_t$ is of type S2.
\end{itemize}

In (S1) and (S2) sets
$\mathcal{X}_t$ are convex, nonempty, and compact sets, see Assumption (H1)-(b) below.
For both kinds of constraints, $\xi_t$ contains in particular the random elements in matrices $A_t, B_t$, and vector $b_t$.
Note that a mix of these types of constraints is allowed: for instance we can have $X_1$ of type S1 and $X_2$ of type $S2$.\\

We make the following assumption on $(\xi_t)$:\\
\par (H0) $(\xi_t)$ 
is interstage independent and
for $t=2,\ldots,T$, $\xi_t$ is a random vector taking values in $\mathbb{R}^K$ with a discrete distribution and a
finite support $\Theta_t=\{\xi_{t 1}, \ldots, \xi_{t M}\}$ with $p_{t i}=\mathbb{P}(\xi_{t}=\xi_{t i})>0,i=1,\ldots,M$,
while $\xi_1$ is deterministic.\footnote{To alleviate notation and without loss of generality, we have assumed that the number $M$ of possible realizations
of $\xi_t$, the size $K$ of $\xi_t$, and $n$ of $x_t$ do not depend on $t$.}\\

We will denote by $A_{t j}, B_{t j},$ and $b_{t j}$ the realizations of respectively $A_t, B_t,$ and $b_t$
in $\xi_{t j}$. For this problem, we can write Dynamic Programming equations: assuming that $\xi_1$ is deterministic,
the first stage problem is 
\begin{equation}\label{firststodp}
\mathcal{Q}_1( x_0 ) = \left\{
\begin{array}{l}
\inf_{x_1 \in \mathbb{R}^n} F_1(x_1, x_0, \xi_1) := f_1(x_1, x_0, \xi_1)  + \mathcal{Q}_2 ( x_1 )\\
x_1 \in X_1( x_{0}, \xi_1 )\\
\end{array}
\right.
\end{equation}
for $x_0$ given and for $t=2,\ldots,T$, 
\begin{equation}\label{defqt}
\mathcal{Q}_t( x_{t-1} )= \mathbb{E}_{\xi_t}[ \mathfrak{Q}_t ( x_{t-1},  \xi_{t}  )  ]
\end{equation}
with
\begin{equation}\label{secondstodp} 
\mathfrak{Q}_t ( x_{t-1}, \xi_{t}  ) = 
\left\{ 
\begin{array}{l}
\inf_{x_t \in \mathbb{R}^n}  F_t(x_t, x_{t-1}, \xi_t ) :=  f_t ( x_t , x_{t-1}, \xi_t ) + \mathcal{Q}_{t+1} ( x_t )\\
x_t \in X_t ( x_{t-1}, \xi_t ),
\end{array}
\right.
\end{equation}
with the convention that $\mathcal{Q}_{T+1}$ is null.

We set $\mathcal{X}_0=\{x_0\}$ and make the following assumptions (H1) on the problem data: for $t=1,\ldots,T$,\\
\par (H1)-(a) for every $j=1,\ldots,M$, the function
$f_t(\cdot, \cdot,\xi_{t j})$ is strongly convex on  $\mathcal{X}_t \small{\times} \mathcal{X}_{t-1}$
with constant of strong convexity $\alpha_{t j}>0$ with respect to norm $\|\cdot\|_2$;
\par (H1)-(b) $\mathcal{X}_t$ is nonempty, convex, and compact;
\par (H1)-(c) there exists $\varepsilon_t>0$ such that for every $j=1,\ldots,M$, for every
$x_{t-1} \in \mathcal{X}_{t-1}^{\varepsilon_t}$,
the set $X_t(x_{t-1}, \xi_{t j}) \cap \mbox{ri}( \mathcal{X}_t)$ is nonempty.\\

If $X_t$ is of type $S2$ we additionally assume that:\\
\par (H1)-(d) for $t=1,\ldots,T$, there exists $\tilde \varepsilon_t>0$ such that for every $j=1,\ldots,M$, each component $g_{t i}(\cdot, \cdot, \xi_{t j}), i=1,\ldots,p$, of the function $g_t(\cdot, \cdot, \xi_{t j})$ is 
convex on $\mathcal{X}_t \small{\times} \mathcal{X}_{t-1}^{\tilde \varepsilon_t}$;
\par  (H1)-(e) for $t=2, \ldots, T$, $\forall j=1,\ldots,M$, 
there exists $({\bar x}_{t j t-1}, {\bar x}_{t j t}) \in \mathcal{X}_{t-1} \small{\times} \mbox{ri}( \mathcal{X}_t )$
such that  $A_{t j} {\bar x}_{t j t} + B_{t j} {\bar x}_{t j t-1} = b_{t j}$,
and $({\bar x}_{t j t-1}, {\bar x}_{t j t}) \in \mbox{ri}( \{ g_t(\cdot, \cdot, \xi_{t j}) \leq 0 \} )$. \\

\begin{rem} Consider a problem of form \eqref{pbtosolve} where the strong convexity assumption of 
functions $f_t(\cdot, \cdot,\xi_{t j})$ fails to hold and for every $t, j$ 
function $f_t(\cdot, \cdot,\xi_{t j})$ is convex and  
the columns
of matrix $( A_{t j} \, B_{t j})$ are independent.
In this situation, we may reformulate the problem pushing and penalizing the linear coupling constraints
 in the objective, ending up with the strongly convex cost function
 $f_t(\cdot, \cdot,\xi_{t j}) + \rho_t \|A_{t j} x_t + B_{t j} x_{t-1} - b_{t j}  \|_2^2$
 in variables $(x_t, x_{t-1})$
 for stage $t$ realization $\xi_{t j}$ for some well chosen penalization $\rho_t >0$.
\end{rem}

We now show in the next section that Assumption (H1)-(a) implies that
recourse functions $\mathcal{Q}_t$
are strongly convex.

\subsection{Strong convexity of recourse functions}

Let $f: X \rightarrow \mathbb{R}$ be a function 
defined on a convex subset $X \subset \mathbb{R}^m$.
Let $X \subset \mathbb{R}^m$ and $Y \subset \mathbb{R}^n$ be two nonempty convex sets.
Let $\mathcal{A}$ be a $p \small{\times} n$ real matrix, let
$\mathcal{B}$ be a $p \small{\times} m$ real matrix, let $f: Y \small{\times} X  \rightarrow \mathbb{R}$,
and let $g:Y \small{\times} X  \rightarrow \mathbb{R}^q$.
For $b \in \mathbb{R}^p$, we define the value function
\begin{equation} \label{optclassgeneral0}
\mathcal{Q}(x)= \left\{
\begin{array}{l}
\inf f(y,x)\\
y \in \mathcal{S}(x):=\{y \in Y, \mathcal{A} y + \mathcal{B} x = b, g(y, x) \leq 0\}.
\end{array}
\right.
\end{equation}
SQDP algorithm 
presented in Section \ref{sec:SQDP} is based on Proposition  \ref{strongconvvfunc} below
giving conditions ensuring that $\mathcal{Q}$ is strongly convex:
\begin{prop}\label{strongconvvfunc} Consider value function $\mathcal{Q}$  given by \eqref{optclassgeneral0}.
Assume that (i) $X, Y$ are nonempty and convex sets such that $X \subseteq \mbox{dom}(\mathcal{Q})$ and
$Y$ is closed, (ii) $f, g$ are lower semicontinuous and the components $g_i$ of $g$ are convex functions.
If additionally $f$ is strongly convex on $Y \small{\times} X$ with constant of strong convexity
$\alpha$ with respect to the Euclidean norm $\|\cdot\|$ on $\mathbb{R}^{m+n}$, then
$\mathcal{Q}$ is strongly convex on $X$ with constant of strong convexity $\alpha$
with respect to the Euclidean norm $\|\cdot\|$ on $\mathbb{R}^{m}$. 
\end{prop}
\begin{proof} Take $x_1, x_2 \in X$ and $\varepsilon>0$.
Since $X \subseteq \mbox{dom}(\mathcal{Q})$ the sets $\mathcal{S}(x_1)$
and $\mathcal{S}(x_2)$ are nonempty.
Our assumptions imply that there are $y_1 \in \mathcal{S}(x_1)$ and  $y_2 \in \mathcal{S}(x_2)$ such that
$\mathcal{Q}( x_1 )\leq  f( y_1, x_1 ) \leq \mathcal{Q}( x_1 )+\varepsilon$ and 
$\mathcal{Q}( x_2 ) \leq f( y_2, x_2 ) \leq \mathcal{Q}( x_2 )+\varepsilon$. Then for every $0 \leq t \leq 1$, by convexity arguments
we have that $t y_1 + (1-t)y_2 \in \mathcal{S}(t x_1 + (1-t) x_2   )$ and therefore
$$
\begin{array}{lll}
\mathcal{Q}(  t x_1 + (1-t) x_2  ) & \leq & f( t y_1 + (1-t) y_2  , t x_1 + (1-t) x_2   )\\
& \leq & t f(y_1 , x_1) + (1-t) f( y_2, x_2 ) - \frac{1}{2} \alpha t (1-t) \| (y_2, x_2)  -   (y_1, x_1)   \|^2 \\
& \leq & t \mathcal{Q}( x_1 ) + (1-t) \mathcal{Q}(x_2) - \frac{1}{2} \alpha t (1-t) \| x_2  -  x_1   \|^2+\varepsilon.
\end{array}
$$
Passing to the limit when $\varepsilon\rightarrow 0$, we obtain
$$
\mathcal{Q}(  t x_1 + (1-t) x_2  ) \leq t \mathcal{Q}( x_1 ) + (1-t) \mathcal{Q}(x_2) - \frac{1}{2} \alpha t (1-t) \| x_2  -  x_1   \|^2,
$$
which completes the proof. \hfill
\end{proof}

\vspace*{0,5cm}

Using Proposition \ref{strongconvvfunc}, we 
easily obtain by induction on $t$ 
that function $\mathfrak{Q}_t(\cdot,\xi_{t j})$
is strongly convex with constant
$\alpha_{t j}$ for norm $\|\cdot\|_2$ which immediately implies, using definition 
\eqref{defqt} of $\mathcal{Q}_t$,
that $\mathcal{Q}_t$ 
is strongly convex for the same norm with
constant of strong convexity
$\sum_{j=1}^M p_{t j}\alpha_{t j}$.

\section{SQDP: extension of QCSC method for strongly convex MSPs}\label{sec:SQDP}

In this section, we introduce SQDP  to solve 
the MSPs with strongly convex cost functions introduced in the previous section. SQDP can be seen as an extension of the popular SDDP method to solve MSPs, which, same as QCSC,
uses quadratic functions (called cuts) to
approximate functions in the objective, instead of the affine cuts used by SDDP.

At iteration $k$,
for $t=2,\ldots,T$,
cost-to-go function $\mathcal{Q}_t$
is approximated by function
$\mathcal{Q}_t^k$ given as a maximum of quadratic functions.

Same as SDDP, SQDP has for each iteration a forward pass followed by a backward pass. In the forward pass,
trial points are obtained sampling a scenario
$(\xi_1,\xi_2^k,\ldots,\xi_T^k)$
for $(\xi_1,\xi_2,\ldots,\xi_T)$
and computing these trial points $x_t^k$, $t=1,\ldots,T-1$, on this scenario replacing the unknown cost-to-go functions
$\mathcal{Q}_{t}$ by 
$\mathcal{Q}_{t}^k$.
In the backward pass, a new quadratic cut is computed for each cost-to-go function
$\mathcal{Q}_t$, $t=2,3,\ldots,T$ at the trial points $x_{t-1}^k$
computed in the forward pass. For short, SQDP is an extension of SDDP where affine lower bounding cuts are replaced by quadratic lower bounding cuts. 

The details of the algorithm, in particular for the computation of the quadratic cuts, is given below.

\noindent\rule[0.5ex]{1\columnwidth}{1pt}

\par  {\textbf{SQDP, Step 1: Initialization.}} For $t=2,\ldots,T$,  take as initial approximations
affine functions
 $\mathcal{Q}_t^0 =\mathcal{C}_{t}^0$
such that $\mathcal{Q}_t^0 \leq \mathcal{Q}_t$ on $\mathcal{X}_{t-1}$.  

Let $x_0$ be given, set the iteration count $k$ to 1, and $\mathcal{Q}_{T+1}^0 =\mathcal{C}_{T+1}^0 \equiv 0$. \\
\par {\textbf{SQDP, Step 2: Forward pass.}}  
\par Sample a realization $\xi_1,\xi_2^k,\ldots,\xi_T^k$ of $\xi_1,\xi_2,\ldots,\xi_T$.
\par {\textbf{For }}$t=1,\ldots,T-1$,
\par $\hspace*{0.75cm}$Compute an optimal solution $x_t^k$ of  
\begin{equation} \label{defxtkj}
{\underline{\mathfrak{Q}}}_t^{k-1}( x_{t-1}^k , \xi_t^k ) = \left\{
\begin{array}{l}
\displaystyle \inf_{x_t} \; F_t^{k-1}(x_t , x_{t-1}^k, \xi_t^k):= f_t( x_t , x_{t-1}^k , \xi_{t}^k ) + \mathcal{Q}_{t+1}^{k-1}( x_t ) \\
x_t \in X_t( x_{t-1}^k, \xi_t^k ),
\end{array}
\right.
\end{equation}
\par  $\hspace*{0.75cm}$where $x_{0}^k = x_0$, $\xi_1^k=\xi_1$.
\par {\textbf{End For}}\\
\par  {\textbf{SQDP, Step 3: Backward pass.}}
\par Set $\theta_{T+1}^k=\alpha_{T+1}=0$ and $\beta_{T+1}^k=0$ which defines cut $\mathcal{C}_{T+1}^k \equiv 0$ and function
$\mathcal{Q}_{T+1}^k\equiv 0$.\\
\par {\textbf{For }}$t=T,\ldots,2$,\\
\hspace*{0.8cm}{\textbf{For }}$j=1,\ldots,M,$\\
\hspace*{1.6cm}Compute an optimal  solution $x_{t j}^{k}$ of
\begin{equation}\label{primalpbisddp}
{\underline{\mathfrak{Q}}}_t^k(x_{t-1}^k , \xi_{t j} ) = \left\{
\begin{array}{l}
\displaystyle \inf_{x_t} \;F_t^k(x_t,x_{t-1}^k,  \xi_{t j}):=f_t( x_t , x_{t-1}^k , \xi_{t j}) + \mathcal{Q}_{t+1}^k ( x_t )\\ 
x_t \in X_t(x_{t-1}^k , \xi_{t j}).\\
\end{array}
\right.
\end{equation}
For the problem above, if $X_t$ is of type $S1$ we define the Lagrangian
$$
L(x_t, \lambda) = F_t^k(x_t , x_{t-1}^k ,  \xi_{t j}) +   \lambda^T (A_{t j} x_t + B_{t j} x_{t-1}^k - b_{t j} )$$
and take optimal Lagrange multipliers $\lambda_{t j}^k$.
If $X_t$ is of type $S2$ we define the Lagrangian
$$L(x_t, \lambda, \mu) = F_t^k(x_t , x_{t-1}^k ,  \xi_{t j}) +   \lambda^T (A_{t j} x_t + B_{t j} x_{t-1}^k - b_{t j} ) + 
\mu^T g_t( x_t , x_{t-1}^k , \xi_{t j})$$ and 
take optimal Lagrange multipliers $(\lambda_{t j}^k, \mu_{t j}^k)$.
If $X_t$ is of type $S1$,  denoting by $$\mbox{SG}_{f_t( x_{t j}^k, \cdot , \xi_{t j} )}( x_{t-1}^{k} )$$
a subgradient of convex function 
$f_t( x_{t j}^{k}, \cdot , \xi_{t j}  )$ at $x_{t-1}^{k}$,
we compute $\theta_{t j}^{k} = {\underline{\mathfrak{Q}}}_t^k(x_{t-1}^k , \xi_{t j} )$
and
$$
\beta_{t j}^{k}=\mbox{SG}_{f_t( x_{t j}^{k} ,\cdot , \xi_{t j}  )}( x_{t-1}^{k} )+ B_{t j}^T \lambda_{t j}^k.
$$
If $X_t$ is of type $S2$ denoting by 
$\mbox{SG}_{g_{t i} (x_{t j}^{k}, \cdot , \xi_{t j} )}(x_{t-1}^{k})$ a subgradient of convex function $g_{t i} (x_{t j}^{k} , \cdot , \xi_{t j} )$
at $x_{t-1}^{k}$ we compute $\theta_{t j}^{k} = {\underline{\mathfrak{Q}}}_t^k(x_{t-1}^k , \xi_{t j} )$ and
$$
\beta_{t j}^{k}=\mbox{SG}_{f_t( x_{t j}^{k} ,\cdot , \xi_{t j}  )}( x_{t-1}^{k} )+ B_{t j}^T \lambda_{t j}^k + \sum_{i=1}^p \mu_{t j}^k(i) \mbox{SG}_{g_{t i} ( x_{t j}^{k} ,\cdot , \xi_{t j}  )}(x_{t-1}^{k}).
$$
\hspace*{0.8cm}{\textbf{End For}}\\
\hspace*{0.8cm}The new cut 
$$
\mathcal{C}_t^k( x_{t-1} ) = \theta_t^{k} +  \langle \beta_t^{k}, x_{t-1}-x_{t-1}^{k} \rangle + \frac{\alpha_t}{2}\| x_{t-1}-x_{t-1}^{k} \|_2^2
$$
\hspace*{0.8cm}is obtained for $\mathcal{Q}_t$ at 
$x_{t-1}^{k}$
computing
\begin{equation}\label{formulathetak}
\theta_t^k=\sum_{j=1}^M p_{t j} \theta_{t j}^{k}, \;\beta_{t}^k=\sum_{j=1}^M p_{t j}  \beta_{t j}^{k}
\mbox{ and }\alpha_t = \sum_{j=1}^M p_{t j} \alpha_{t j},
\end{equation}
\hspace*{0.8cm}with the new cost-to-go function for stage $t$ being
\begin{equation}\label{updaterecf}
\mathcal{Q}_t^k=\max(\mathcal{Q}_{t}^{k-1},\mathcal{C}_t^k).
\end{equation}
{\textbf{End For}}\\
\par {\textbf{SQDP, Step 4:}} Do $k \leftarrow k+1$ and go to Step 2.\\
\noindent\rule[0.5ex]{1\columnwidth}{1pt}

We will show in the next section the validity of the cuts $\mathcal{C}_t^k$ computed, i.e., that these cuts are lower bounding functions for $\mathcal{Q}_t$. The model update formula \eqref{updaterecf}
shows that the approximation $\mathcal{Q}_t^k$ of
$\mathcal{Q}_t$ is the maximum of all cuts
$\mathcal{C}_t^1,\mathcal{C}_t^2,\ldots,\mathcal{C}_t^k$ computed for $\mathcal{Q}_t$  up to iteration $k$.

\begin{rem} By change of variable we can write $\mathcal{Q}_t^k$ under the form
$$
\mathcal{Q}_t^k ( x_{t-1} ) = \frac{\alpha_t}{2}\|x_{t-1} - {\bar x}_{t-1}^k\|_2^2 + \max_{1 \leq j \leq k} \;{\tilde \theta}_t^{j}
+ \langle \tilde \beta_t^{j} , x_{t-1} \rangle.
$$
Therefore all subproblems can be reformulated as quadratic programs with a quadratic term multiple of $\|x-\bar x\|_2^2$
and linear constraints (the same number as with SDDP).
\end{rem}

\section{Convergence analysis of SQDP}\label{convanalysis}

In this section we prove the convergence of SQDP. An important ingredient in this proof is to show that the quadratic cuts computed by 
SQDP are valid cuts, i.e., lower bounding functions for cost-to-go functions $\mathcal{Q}_t$.
This is shown in the next proposition
\ref{sqdpvalidcuts}.

To prove this proposition, we will use the following equivalent definition
of strong convexity:
\begin{definition}\label{defsc2}
Let $X \subset \mathbb{R}^m$ be a convex set.
Function $f: X \rightarrow \mathbb{R}$
is strongly convex on $X$ with constant of strong convexity $\alpha > 0$ 
with respect to norm $\|\cdot\|$
if and only {if}
\begin{equation}\label{characscf}
f(y) \geq f(x) + s^T (y-x)   + \frac{\alpha}{2}\|y-x\|^2, \;\forall x, y \in X, \forall s \in \partial f(x).
\end{equation}
\end{definition}
Definition \ref{defsc2} gives inequality \eqref{eq:qf}
used in QCSC.

\begin{prop}\label{sqdpvalidcuts} For every $t=2,\ldots,T+1$, for every $k \geq 1$, we have for every $x_{t-1} \in \mathcal{X}_{t-1}$:
\begin{equation}\label{validcut0}
\mathcal{Q}_t(x_{t-1}) \geq \mathcal{C}_t^k(x_{t-1}),\;
\mathcal{Q}_t(x_{t-1}) \geq \mathcal{Q}_t^k(x_{t-1}),\;
\mathcal{Q}_t^k \geq \mathcal{Q}_t^{k-1}.
\end{equation}
\end{prop}
\begin{proof}
The relation
$\mathcal{Q}_t^k \geq \mathcal{Q}_t^{k-1}$
is an immediate consequence of the model update step \eqref{updaterecf}.
We now show 
\begin{equation}\label{validcut}
\mathcal{Q}_t(x_{t-1}) \geq \mathcal{C}_t^k(x_{t-1}),\;\; \mathcal{Q}_t(x_{t-1}) \geq \mathcal{Q}_t^k(x_{t-1})
\end{equation}
by forward induction on $k$ and backward induction on $t$.

Note that relations \eqref{validcut}
hold for $k=0$ and
$t=2,\ldots,T+1$.
Assume now that \eqref{validcut}
holds until $k-1$ and for
$t=2,\ldots,T+1$.
Inequality \eqref{validcut} holds for $t=T+1$ (in this case both functions 
$\mathcal{Q}_{T+1}^k$
and $\mathcal{C}_{T+1}^k$ are null) and if it holds for 
$t+1$ 
with  $t \in \{2,\ldots,T\}$, we deduce that for any $x_{t-1} \in \mathcal{X}_{t-1}$ and $x_{t} \in \mathcal{X}_{t}$,  for every $j=1,\ldots,M$, we have
$\mathcal{Q}_{t+1}( x_t ) \geq \mathcal{Q}_{t+1}^k ( x_t )$, 
$\mathfrak{Q}_t( x_{t-1} , \xi_{t j} ) \geq  {\underline{\mathfrak{Q}}}_t^{k}( x_{t-1} , \xi_{t j} )$.
Now note that function $x_t \rightarrow \mathcal{Q}_{t+1}^k(x_t)$ is convex (as a maximum of convex functions) and recalling
that $(x_t,x_{t-1}) \rightarrow f_t(x_t,x_{t-1},\xi_{t j})$ is strongly convex with constant of strong convexity
$\alpha_{t j}$, the function $(x_t,x_{t-1}) \rightarrow f_t(x_t,x_{t-1},\xi_{t j}) +\mathcal{Q}_{t+1}^k( x_t )$ is also
strongly convex with the same parameter of strong convexity. Using Proposition \ref{strongconvvfunc},
it follows that ${\underline{\mathfrak{Q}}}_t^{k}( \cdot , \xi_{t j} )$  is strongly convex  
with constant of strong convexity $\alpha_{t j}$.
Using Lemma 2.1 in Guigues (2016) we have that
$\beta_{t j}^{k} \in \partial  {\underline{\mathfrak{Q}}}_t^k(\cdot , \xi_{t j} )(x_{t-1}^k)$.
Recalling characterization \eqref{characscf} of strongly convex functions,
we get for any $x_{t-1} \in \mathcal{X}_{t-1}$:
\begin{align}
{\underline{\mathfrak{Q}}}_t^{k}( x_{t-1} , \xi_{t j} )  \geq   {\underline{\mathfrak{Q}}}_t^{k}( x_{t-1}^k , \xi_{t j} )+ \langle \beta_{t j}^{k} , x_{t-1} - x_{t-1}^k \rangle + \frac{\alpha_{t j}}{2} \|x_{t-1}- x_{t-1}^k\|_2^2
\end{align} 
and therefore for any $x_{t-1} \in \mathcal{X}_{t-1}$, we have 
\begin{align}\label{checkvalidcut}
\mathcal{Q}_t( x_{t-1} ) & =  \displaystyle \sum_{j=1}^M p_{t j} \mathfrak{Q}_t( x_{t-1} , \xi_{t j} )\\
& \geq   \displaystyle \sum_{j=1}^M p_{t j} {\underline{\mathfrak{Q}}}_t^{k}( x_{t-1} , \xi_{t j} )\\
& \geq  \displaystyle \sum_{j=1}^M p_{t j} \Big(    {\underline{\mathfrak{Q}}}_t^{k}( x_{t-1}^k , \xi_{t j} )
+ \langle \beta_{t j}^{k} , x_{t-1} - x_{t-1}^k \rangle + \frac{\alpha_{t j}}{2} \|x_{t-1} -x_{t-1}^k\|_2^2  \Big)\\
& =  \theta_t^{k} +  \langle \beta_t^{k}, x_{t-1}-x_{t-1}^k \rangle + \frac{\alpha_t}{2}\| x_{t-1}-x_{t-1}^k \|_2^2 \\
& =  \mathcal{C}_t^k (x_{t-1}). 
\end{align} 
Using $\mathcal{Q}_t \geq 
\mathcal{Q}_t^{k-1}$ and 
$\mathcal{Q}_t^k = \max(\mathcal{Q}_t^{k-1},\mathcal{C}_t^k)$, we have
$\mathcal{Q}_t \geq \mathcal{Q}_t^k$, 
which completes the induction step and shows \eqref{validcut} for every $t, k$.
\end{proof}

\vspace*{0.2cm}
We now define a simulation of the
SQDP policy which will compute
decisions for all possible realizations of the uncertainties. For that, we need more notation.

Due to Assumption (H0), the $M^{T-1}$ realizations of $(\xi_t)_{t=1}^T$ form a scenario tree of depth $T+1$
where the root node $n_0$ associated to a stage $0$ (with decision $x_0$ taken at that
node) has one child node $n_1$
associated to the first stage (with $\xi_1$ deterministic).\\
We denote by $\mathcal{N}$ the set of nodes, by
{\tt{Nodes}}$(t)$ the set of nodes for stage $t$ and
for a node $n$ of the tree, we define: 
\begin{itemize}
\item $C(n)$: the set of children nodes (the empty set for the leaves);
\item $x_n$: a decision taken at that node;
\item $p_n$: the transition probability from the parent node of $n$ to $n$;
\item $\xi_n$: the realization of process $(\xi_t)$ at node $n$\footnote{The same notation $\xi_{\tt{Index}}$ is used to denote
the realization of the process at node {\tt{Index}} of the scenario tree and the value of the process $(\xi_t)$
for stage {\tt{Index}}. The context will allow us to know which concept is being referred to.
In particular, letters $n$ and $m$ will only be used to refer to nodes while $t$ will be used to refer to stages.}:
for a node $n$ of stage $t$, this realization $\xi_n$ contains in particular the realizations
$b_n$ of $b_t$, $A_{n}$ of $A_{t}$, and $B_{n}$ of $B_{t}$;
\item $\xi_{[n]}$: the history of the realizations of process $(\xi_t)$ from the first stage node $n_1$ to node $n$:
 for a node $n$ of stage $t$, the $i$-th component of $\xi_{[n]}$ is $\xi_{\mathcal{P}^{t-i}(n)}$ for $i=1,\ldots, t$,
 where $\mathcal{P}:\mathcal{N} \rightarrow \mathcal{N}$ is the function 
 associating to a node its parent node (the empty set for the root node).
\end{itemize}

With this notation, we now provide  the simulation of 
SQDP policy at iteration $k$ which computes
decisions $x_n^k$ for all nodes
$n$ of the scenario tree.\\

\par {\textbf{SQDP, simulation of the policy at iteration $k$.}} 
\par {\textbf{For }}$t=1,\ldots,T$,\\
\hspace*{1.6cm}{\textbf{For }}every node $n$ of stage $t-1$,\\
\hspace*{2.4cm}{\textbf{For }}every child node $m$ of node $n$, compute an optimal solution $x_m^k$ of
\begin{equation} \label{defxtkj}
{\underline{\mathfrak{Q}}}_t^{k-1}( x_n^k , \xi_m ) = \left\{
\begin{array}{l}
\displaystyle \inf_{x_m} \; F_t^{k-1}(x_m , x_n^k, \xi_m):= f_t( x_m , x_n^k , \xi_m ) + \mathcal{Q}_{t+1}^{k-1}( x_m ) \\
x_m \in X_t( x_n^k, \xi_m ),
\end{array}
\right.
\end{equation}
\hspace*{2.6cm}where $x_{n_0}^k = x_0$.\\
\hspace*{2.4cm}{\textbf{End For}}\\
\hspace*{1.6cm}{\textbf{End For}}
\par {\textbf{End For}}

\vspace*{0.2cm}

In Theorem \ref{convsddpsconv} below we show the convergence of SQDP making the following additional assumption:\\

\par (H2) The samples in the forward passes are independent:
$\xi^k=(\xi_2^k,\ldots,\xi_T^k)$ is a realization of
$(\xi_2,\ldots,\xi_T)$, and $\xi^1,\xi^2,\ldots$ are independent. \\

\par We will make use of the following lemma:
\begin{lemma} \label{convrecfuncQtS} Let Assumptions (H0) and (H1) hold. Then for $t=2,\ldots,T+1$, function $\mathcal{Q}_t$ is convex and Lipschitz continuous on 
$\mathcal{X}_{t-1}$.
\end{lemma}
\begin{proof} The proof is analogue to the proofs of Lemma 3.2 in Guigues (2016) and
Lemma 2.2 in Girardeau et al. (2015).\hfill
\end{proof}

\vspace*{0.2cm}
The convergence of SQDP is given in the following theorem.

\begin{thm} \label{convsddpsconv}
Consider the sequences of recourse functions $\mathcal{Q}_ t^k$
generated by SQDP and the decisions $x_n^k$ computed by the simulation of SQDP policy.
Let Assumptions (H0), (H1) and (H2) hold. Then
\begin{itemize}
\item[(i)] almost surely, for $t=2,\ldots,T+1$, the following holds:
$$
\mathcal{H}(t): \;\;\;\forall n \in {\tt{Nodes}}(t-1), \;\; \displaystyle \lim_{k \rightarrow +\infty} \mathcal{Q}_{t}(x_{n}^{k})-\mathcal{Q}_{t}^{k}(x_{n}^{k} )=0.
$$
\item[(ii)]
Almost surely, the limit of the sequence
$( {F}_1^{k-1}(x_{n_1}^k , x_0 ,  \xi_1) )_k$ of the approximate first stage optimal values
and of the sequence
$({\underline{\mathfrak{Q}}}_1^{k}(x_{0}, \xi_1))_k$
is the optimal value 
$\mathcal{Q}_{1}(x_0)$ of \eqref{pbtosolve}.
Let $\Omega=(\Theta_2 \small{\times} \ldots \small{\times} \Theta_T)^{\infty}$ be the sample space
of all possible sequences of scenarios equipped with the product $\mathbb{P}$ of the corresponding 
probability measures. Define on $\Omega$ the random variable 
$x^* = (x_1^*, \ldots, x_T^*)$ as follows. For $\omega \in \Omega$, consider the 
corresponding sequence of decisions $( (x_n^k( \omega ))_{n \in \mathcal{N}} )_{k \geq 1}$
computed by SQDP. Take any accumulation point
$(x_n^* (\omega) )_{n \in \mathcal{N}}$ of this sequence. If $\mathcal{Z}_t$ is the set of $\mathcal{F}_t$-measurable $\mathbb{R}^n$-valued  functions,
define $x_1^*(\omega),\ldots,x_T^*(\omega)$ with $x_t^{*}(\omega)\in \mathcal{Z}_t$  given by
$x_t^{*}(\omega)( \xi_1, \ldots, \xi_t  )=x_{m}^{*}(\omega)$ where $m$ is given by $\xi_{[m]}=(\xi_1,\ldots,\xi_t)$ for $t=1,\ldots,T$.
Then $\mathbb{P}((x_1^*,\ldots,x_T^*) \mbox{ is an optimal solution to \eqref{pbtosolve}})  =1$.

\end{itemize}
\end{thm}
\begin{proof} Let us prove (i).  Let $\Omega_1$ be the event on the sample space $\Omega$  of sequences
of scenarios such that every scenario is sampled an infinite number of times.
Due to (H2), this event has probability one.
Take an arbitrary  realization $\omega$ of SQDP in $\Omega_1$. 
To simplify notation we will use $x_n^k, \mathcal{Q}_t^k, \theta_t^k, \beta_t^k$ instead 
of $x_n^k(\omega), \mathcal{Q}_t^k(\omega), \theta_t^k(\omega), \beta_t^k( \omega )$.\\
We want to show that $\mathcal{H}(t), t=2,\ldots,T+1$, hold for that realization.
The proof is by backward induction on $t$. For $t=T+1$, $\mathcal{H}(t)$ holds
by definition of $\mathcal{Q}_{T+1}$, $\mathcal{Q}_{T+1}^k$. Now assume that $\mathcal{H}(t+1)$ holds
for some $t \in \{2,\ldots,T\}$. We want to show that $\mathcal{H}(t)$ holds.
Take an arbitrary node $n \in {\tt{Nodes}}(t-1)$. For this node we define 
$\mathcal{S}_n$ the set of iterations such that the sampled scenario passes through node $n$.
Observe that $\mathcal{S}_n$ is infinite because the realization of SQDP is in $\Omega_1$.
We first show that 
$$
\displaystyle \lim_{k \rightarrow +\infty, k \in \mathcal{S}_n } \mathcal{Q}_{t}(x_{n}^{k})-\mathcal{Q}_{t}^{k}(x_{n}^{k} )=0. 
$$
For $k \in \mathcal{S}_n$, we have $x_n^k = x_{t-1}^k$, which implies, using \eqref{validcut0} and the definition of $\mathcal{C}_t^k$, that
\begin{equation}\label{firsteqstosddp}
\mathcal{Q}_t ( x_n^k ) \geq \mathcal{Q}_t^k ( x_n^k ) \geq  \mathcal{C}_t^k ( x_n^k ) = \theta_t^k = \sum_{m \in C(n)} p_m  {\underline{\mathfrak{Q}}}_t^k(x_n^k , \xi_m )
\end{equation}
by definition of $\mathcal{C}_t^k$ and $\theta_t^k$.
It follows that for any $k \in \mathcal{S}_n$ we have
\begin{equation} \label{eqconv1bisfutures}
\begin{array}{lll}
0 \leq \mathcal{Q}_{t}(x_{n}^k) - \mathcal{Q}_{t}^k(x_{n}^k) & \leq & 
\displaystyle \sum_{m \in C(n)} p_m \Big( \mathfrak{Q}_t(x_n^k , \xi_m )  -    {\underline{\mathfrak{Q}}}_t^k(x_n^k , \xi_m )    \Big) \\ 
&  \leq & 
\displaystyle \sum_{m \in C(n)} p_m \Big( \mathfrak{Q}_t(x_n^k , \xi_m )  -    {\underline{\mathfrak{Q}}}_t^{k-1}(x_n^k , \xi_m )    \Big)  \\
&  = & 
\displaystyle \sum_{m \in C(n)} p_m \Big( \mathfrak{Q}_t(x_n^k , \xi_m )  -   F_t^{k-1}(x_m^{k}, x_{n}^k, \xi_m )     \Big)  \\ 
&  = & 
\displaystyle \sum_{m \in C(n)} p_m \Big( \mathfrak{Q}_t(x_n^k , \xi_m )  -   f_t(x_m^{k}, x_{n}^k, \xi_m ) -  \mathcal{Q}_{t+1}^{k-1}( x_m^k )     \Big)  \\ 
&  = & 
\displaystyle \sum_{m \in C(n)} p_m \Big( \mathfrak{Q}_t(x_n^k , \xi_m )  -   F_t(x_m^{k}, x_{n}^k, \xi_m ) + \mathcal{Q}_{t+1}( x_m^k )  -  \mathcal{Q}_{t+1}^{k-1}( x_m^k )     \Big)  \\ 
 & \leq  & \displaystyle \sum_{m \in C(n)} p_m \Big(  \mathcal{Q}_{t+1}( x_m^k ) - \mathcal{Q}_{t+1}^{k-1}( x_m^k ) \Big),
\end{array}
\end{equation}
where for the last inequality we have used the definition of $\mathfrak{Q}_t$  and the fact that $x_m^k \in X_t ( x_n^k , \xi_m )$.

Next, recall that $\mathcal{Q}_{t+1}$ is convex; by Lemma \ref{convrecfuncQtS} functions $(\mathcal{Q}_{t+1}^k)_k$ are Lipschitz continuous;
and for all $k \geq 1$ we have $\mathcal{Q}_{t+1}^k \leq \mathcal{Q}_{t+1}^{k+1} \leq\mathcal{Q}_{t+1}$ on compact set $\mathcal{X}_t$.
Therefore, the induction hypothesis
$$
\lim_{k \rightarrow +\infty} \mathcal{Q}_{t+1}( x_m^k ) - \mathcal{Q}_{t+1}^k ( x_m^k )=0 
$$
implies, using Lemma A.1 in Girardeau et al. (2015), that
\begin{equation}\label{indchypreformulatedsto}
\lim_{k \rightarrow +\infty} \mathcal{Q}_{t+1}( x_m^k ) - \mathcal{Q}_{t+1}^{k-1} ( x_m^k )=0 . 
\end{equation}

Plugging \eqref{indchypreformulatedsto} into
\eqref{eqconv1bisfutures} we obtain 
\begin{equation}\label{mainisddp}
\displaystyle \lim_{k \rightarrow +\infty, k \in \mathcal{S}_n} \mathcal{Q}_{t}(x_{n}^{k})-\mathcal{Q}_{t}^{k}(x_{n}^{k} )=0. 
\end{equation}
It remains to show that
\begin{equation}\label{alsoconvnotinsnisddp}
\displaystyle \lim_{k \rightarrow +\infty, k \notin \mathcal{S}_n } \mathcal{Q}_{t}(x_{n}^{k})-\mathcal{Q}_{t}^{k}(x_{n}^{k} )=0. 
\end{equation}
The relation \eqref{alsoconvnotinsnisddp} above can be proved using Lemma 5.4 in Guigues et al. (2020) which can be applied since 
(A) relation \eqref{mainisddp} holds (convergence was shown for the iterations in $\mathcal{S}_n$),
(B) the sequence $(\mathcal{Q}_t^k)_k$ is monotone, i.e., 
$\mathcal{Q}_t^k \geq \mathcal{Q}_t^{k-1}$ for all $k \geq 1$, (C) Assumption (H2) holds, and
(D) $\xi_{t-1}^k$ is independent on $( (x_{n}^j,j=1,\ldots,k), (\mathcal{Q}_{t}^j,j=1,\ldots,k-1))$.\footnote{Lemma 5.4 in Guigues et al. (2020) is similar to the end of the proof of Theorem 4.1 in Guigues (2016) and uses
the Strong Law of Large Numbers. This lemma itself applies the ideas of the end of the convergence proof of SDDP given in Girardeau (2015), which
was given with a different (more general) sampling scheme in the backward pass.} Therefore, we have shown (i).\\

\par (ii) can be proved as Theorem 5.3-(ii) in Guigues (2016) using (i).\hfill
\end{proof}

\section{Numerical experiments comparing SDDP and SQDP}\label{sec:num}

\subsection{Multistage stochastic programming problems}

Consider the Dynamic Programming equations
given by
$\mathcal{Q}_t( x_{t-1} )=\mathbb{E}_{\xi_t}[\mathfrak{Q}_t(x_{t-1},\xi_t)]$
for $t=1,\ldots,T$, $x_0$ given,
$\mathcal{Q}_{T+1} \equiv 0$, where for
$t=1,\ldots,T$, we have
$$
\mathfrak{Q}_t(x_{t-1},\xi_t) =
\left\{ 
\begin{array}{l}
\displaystyle \inf_{x_t \in \mathbb{R}^n} 
\frac{1}{2}\left( 
\begin{array}{c}
x_{t-1}\\
x_t
\end{array}
\right)^T \left(\xi_t \xi_t^T + \lambda_0 I \right)
\left( 
\begin{array}{c}
x_{t-1}\\
x_t
\end{array}
\right) + 
\xi_t^T \left( 
\begin{array}{c}
x_{t-1}\\
x_t
\end{array}
\right) + \mathcal{Q}_{t+1} ( x_t )\\
x_t \geq 0, \displaystyle \sum_{i=1}^n x_t(i)=1.
\end{array}
\right.
$$

Recourse functions 
$\mathcal{Q}_t$ are strongly convex
with parameter $\lambda_0$ for $\|\cdot\|_2$ and therefore
SQDP and SDDP can be applied to solve this problem.
We run SQDP and SDDP
for several values of 
$T$, $n$ (the common size
of $x_{t-1}$ and $x_t$), $M$ (the number of possible realizations
of $\xi_t$ for every $t$),
and $\lambda_0$. Entries in realizations $\xi_{ti}$ of $\xi_t$ are drawn independently
from the uniform distribution in $[0,1]$.
As in SDDP, we can compute
for SQDP at every iteration a lower
bound LB on the optimal value of the
problem which is the optimal value
of the approximate first stage problem
obtained replacing $\mathcal{Q}_2$
by its approximation $\mathcal{Q}_2^k$.
Similarly to SDDP, we also compute
with SQDP a statistical upper bound UB. 
All iterations are run with
a single scenario in the forward pass.
We use the stopping test
from the numerical experiments
in Guigues et al. (2023)
computing the statistical upper bound from
the total cost (from $t=1$
to $t=T$) on the last 200 scenarios:
the algorithm stops when
(UB-LB)/UB$\leq \varepsilon$
with $\varepsilon=0.1$.

We report in Table \ref{tableres}
the CPU time (in seconds) needed to solve the problem and the lower (LB in the table) and upper (UB in the table) bounds
at termination. 

\begin{table}[H]
\centering
\begin{tabular}{|c|c|c|c|c|c|c|c|}
 \hline
 T &  n & M & $\lambda_0$ & \begin{tabular}{c}SQDP\\ (CPU)\end{tabular} & \begin{tabular}{c}SDDP\\ (CPU)\end{tabular} & SQDP LB/UB & SDDP LB/UB\\
 \hline
10 &50 &10&$10^3$& 704& 1191 &25 474/26 161&25 554/26 521\\
 \hline
5 &50&2&0.1&273&315 &354.72/359.15&354.83/359.31\\
 \hline
 7 &50&2&0.1&463&527&356.03/365.11&355.77/364.25\\
 \hline
  3 &50&2&1&136&155 &376.79/378.55&377.51/378.86\\
 \hline
 4&100&5&$10^5$&84&418 &5.02x$10^5$/5.04x$10^5$&5.02x$10^5$/5.04x$10^5$\\
 \hline
  4&100&5&$10^6$&71&162 &5.004x$10^7$/5.008x$10^7$&5.004x$10^7$/5.009x$10^7$\\
 \hline
  3&500&5&$10^6$&543&1560 &2.5004x$10^8$/2.5004x$10^8$&2.5004x$10^8$/2.5004x$10^8$\\
 \hline
 3&600&5&$10^6$&984&4116 &3.0005x$10^8$/3.0005x$10^8$&3.0005x$10^8$/3.0005x$10^8$\\
 \hline
  3&200&5&$10^5$&141&791 &1.0006x$10^7$/1.0008.x$10^7$&1.0006.x$10^7$/1.0008x$10^7$\\
 \hline
\end{tabular}
\caption{CPU time (in seconds) and upper and lower bounds at termination for several instances with SQDP and SDDP}\label{tableres}
\end{table}

To complement Table \ref{tableres}, we also ran 10 independent instances for each parameter configuration. For each fixed tuple $(T,n,M,\lambda_0)$, Table \ref{table_sqdp_sddp_10_instances} reports the average CPU time of SQDP and SDDP, where the average is taken over all
CPU  times of the 10 runs. This provides a more robust comparison of the computational performance of both methods.

\begin{table}[H]
\centering
\begin{tabular}{|c|c|c|c|c|c|}
\hline
$T$ & $n$ & $M$ & $\lambda_0$ & SQDP (CPU) & SDDP (CPU) \\
\hline
10 & 50  & 10 & $10^3$ & 750 & 1262 \\
\hline
5 & 50  & 2 & $0.1$    & 293 & 303 \\
\hline
7&50&2&0.1&420&537 \\
\hline
3 & 50 & 2 & 1    & 127 & 149\\
\hline
4  & 100 & 5  & $10^5$ & 65 & 587 \\
\hline
4  & 100 & 5  & $10^6$ & 38 & 188 \\
\hline
3  & 500 & 5  & $10^6$ & 454 & 2181 \\
\hline
3  & 600 & 5  & $10^6$ & 636 & 2938 \\
\hline
3  & 200 & 5  & $10^5$ & 105 & 505  \\
\hline
\end{tabular}

\caption{Average CPU time (in seconds) over 10 independent instances for each parameter configuration.}\label{table_sqdp_sddp_10_instances}

\end{table}
The Matlab code of this implementation is available
at \url{https://github.com/vguigues/DASC}
on github. In this implementation, linear and quadratic subproblems in SQDP and SDDP were solved using Gurobi.
The methods were run on an Intel Core i7, 1.8GHz, processor with 12,0 Go of
RAM. We observe convergence of the upper and lower
bounds to the optimal value and at termination close approximate
optimal values for SDDP and SQDP (as a check for
correctness of the implementations).
The CPU times are much smaller with SQDP when the parameter $\lambda_0$
is large; otherwise they are of the same 
order of magnitude.
We indeed expect
that when the constant of strong convexity of $\mathcal{Q}_t$ is large ($\lambda_0$ for our instance) SQDP provides much better approximations of these functions
and will converge quicker (providing "better" approximations and therefore "good" trial points quicker).

\subsection{Two-stage stochastic programs}\label{two_stage_example}

We now consider a practical application of the quadratic cutting-plane strategy proposed in SQDP. Specifically, we study a \emph{two-stage} stochastic linear recourse problem whose expected recourse function is strongly convex on a suitable open convex set. Our goal is to compare the classical stochastic decomposition method \cite{stochasti_decomposition_higle_sen_91} based on affine cuts with a quadratic-cut variant that exploits a valid strong convexity constant.

The example is a simplified component-level recourse model inspired by the multiproduct assembly model of Section~1.3.1 in \cite{shadenrbook}. In this setting, a random product demand is mapped into component requirements through the product structure matrix, and the second-stage cost is defined directly in terms of component imbalance. All quantities in this section are expressed in normalized lot units to simplify the presentation. The resulting problem has the form
\begin{equation}\label{defQ2}
 \min_{x\in X}\; c^\top x + \mathcal{Q}_2(x),
 \qquad \mathcal{Q}_2(x)=\mathbb{E}[\mathfrak{Q}_2(x,D)], 
\end{equation}
where \(x\) is the first-stage component-ordering vector, \(c^\top x\) is the ordering cost, and \(\mathcal{Q}_2\) is the expected recourse cost induced by the random demand vector \(D\) which has continuous distribution. Furthermore, $A^\top D - x$ is the component imbalance vector and $\varphi$ is the recourse penalty function.

To solve this problem, we use the Stochastic Decomposition method of \cite{stochasti_decomposition_higle_sen_91}, which iteratively builds lower approximations of the expected recourse function from sampled scenarios and dual information.  All computations were performed in \textsc{Matlab} using \textsc{Gurobi}.

\subsubsection{A simplified two-stage assembly recourse model}\label{two_stage_example_assembly}

For testing purposes, we consider a small-scale instance with two products and two shared components. The product structure matrix, where product \(1\) consumes one unit of component \(B\) and two units of component \(C\), while product \(2\) consumes three units of component \(B\) and one unit of component \(C\), is
\[
A=
\begin{pmatrix}
1 & 2\\
3 & 1
\end{pmatrix}.
\]

The first-stage decision vector is \(x=(x_B,x_C)\in X=[1.80,2.70]\times[1.35,1.95]\), where \(x_B\) and \(x_C\) denote the quantities of components ordered before uncertainty is revealed. In this problem, \(X\) should be interpreted as a plausible operational range of admissible component orders (to ensure strong
convexity of $\mathcal{Q}_2$
set $X$ cannot contain the origin).

The product demand is modeled by \(D=(D_1,D_2)\sim \mathcal{U}([0,1]^2)\)
with uniform distribution, and the induced component demand is \(Z=A^\top D\). So, we obtain \(Z_B=D_1+3D_2\) and \(Z_C=2D_1+D_2\). The first-stage cost is 
$$c^\top x=1.8x_B+0.8x_C.$$

The second stage cost-function $\mathfrak{Q}_2$
in \eqref{defQ2} is given by
\[
\begin{array}{lcl}
\mathfrak{Q}_2(x,d)
& = &
\min_{u,v\ge 0}
\left\{
\eta^\top u + \pi^\top v
\;:\;
u-v=A^\top d-x
\right\},\\
& = & \eta_B[d_1+3d_2-x_B]^+
+\eta_C[2d_1+d_2-x_C]^+\\
&&
+\pi_B[-d_1-3d_2+x_B]^+
+\pi_C[-2d_1-d_2+x_C]^+,
\end{array}
\]
where \(u\) and \(v\) represent, respectively, component shortages and surpluses, $\eta=(\eta_B,\eta_C)$,
$\pi=(\pi_B,\pi_C)$ and $d=(d_1,d_2)$ denotes a realization of $D$. We take shortage
penalties
$\eta_B=12$ and $\eta_C=7$,
and surplus penalties
$\pi_B=3$ and $\pi_C=3.5$.

We consider the open convex set \(V=(1.75,2.75)\times(1.30,2.00)\), which contains \(X\). Since \(\det(A^\top)=-5\), the transformed random vector \(Z=A^\top D\) has constant density \(r=1/5\) on its support. With \(\rho\approx 0.10365\), the \(\rho\)-neighborhood of \(V\) remains inside the support of \(Z\), and by \cite[Theorem~2.2]{SCHULTZ19943}, \(\mathcal{Q}_2\) is strongly convex on \(V\). The assumptions of the two-stage stochastic decomposition framework in \cite{stochasti_decomposition_higle_sen_91} are also satisfied, so the method applies to this problem.

For the adapted stochastic decomposition method with quadratic cuts, a valid strong convexity constant is required. In this example, we compute this constant ($\kappa$) using the Hessian characterization of strong convexity \cite[Lemma~2.3]{CLAUS_Spurkel_22}. For this model, the Hessian of \(\mathcal{Q}_2\) is diagonal on \(V\) and we obtain $\kappa=5.0$.

The results in Table~\ref{tab:assembly_example_results} correspond to a single run of each method. In this experiment, the quadratic approach achieved the best overall performance: it required fewer iterations, the smallest CPU time, and the lowest estimated expected total cost. This means that it was both computationally faster and better in terms of out-of-sample objective value. The linear method was slower and produced a slightly larger estimated expected total cost.

\begin{table}[ht]
\centering

\small
\setlength{\tabcolsep}{5pt}
\renewcommand{\arraystretch}{1.15}
\begin{tabular}{lccc}
\hline
Method & Iterations & CPU (s) &
\begin{tabular}[c]{@{}c@{}}
Estimated Expected\\
Total Cost
\end{tabular}
\\
\hline
Stochastic Decomposition
& 101
& 0.598025
& 12.76303905
\\
SQDP
& 24
& 0.158909
& 12.73863300
\\
\hline
\end{tabular}
\caption{Computational results for Stochastic Decomposition 
and SQDP
methods. The reported estimated expected total cost is an out-of-sample Monte Carlo estimate of the true objective value.}
\label{tab:assembly_example_results}
\end{table}

We can check the correct implementation
of the Stochastic Decomposition and
SQDP methods on this example
observing that  for this application, we 
have an analytic expression
of function $\mathcal{Q}_2$.
Indeed, we easily check that
$$
\mathcal{Q}_2(x) = \mathcal{Q}_2^1(x_B) +
\mathcal{Q}_2^2(x_B)+\mathcal{Q}_2^3(x_C)
+\mathcal{Q}_2^4(x_C)
$$
where 
$$
\displaystyle
\mathcal{Q}_2^1(x_B)=
\eta_B \mathbb{E}[D_1+3 D_2-x_B]^+
=\eta_B 
\left\{ 
\begin{array}{ll}
\displaystyle 2-x_B & x_B \leq 0,\\
\displaystyle 2-x_B+\frac{x_B^3}{18},& 0 \leq x_B \leq 1,\\
\displaystyle \frac{x_B^2}{6}-\frac{7}{6}x_B + \frac{37}{18},& 1 \leq x_B \leq 3,\\ 
\displaystyle \frac{(4-x_B)^3}{18},& 3 \leq x_B \leq 4,\\
0 & x_B \geq 4;
\end{array}
\right.
$$

$$
\displaystyle \mathcal{Q}_2^2(x_B)=
\pi_B \mathbb{E}[-D_1-3 D_2+x_B]^+
=\pi_B 
\left\{ 
\begin{array}{ll}
0 & x_B \leq 0,\\
\displaystyle \frac{x_B^3}{18},& 0 \leq x_B \leq 1,\\
\displaystyle \frac{x_B^2}{6}-\frac{1}{6}x_B + \frac{1}{18},& 1 \leq x_B \leq 3,\\ 
\displaystyle -\frac{x_B^3}{18}+\frac{2}{3}x_B^2-\frac{5}{3}x_B + \frac{14}{9},& 3 \leq x_B \leq 4,\\
\displaystyle x_B-2 & x_B \geq 4;
\end{array}
\right.
$$

$$
\displaystyle \mathcal{Q}_2^3(x_C)=
\eta_C \mathbb{E}[2D_1+ D_2-x_C]^+
=\eta_C 
\left\{ 
\begin{array}{ll}
\displaystyle \frac{3}{2}-x_C & x_C \leq 0,\\
\displaystyle \frac{x_C^3}{12}-x_C+\frac{3}{2},& 0 \leq x_C \leq 1,\\
\displaystyle \frac{x_C^2}{4}-\frac{5}{4}x_C + \frac{19}{12},& 1 \leq x_C \leq 2,\\ 
\displaystyle -\frac{x_C^3}{12}+\frac{3}{4}x_C^2-\frac{9}{4}x_C + \frac{9}{4},& 2 \leq x_C \leq 3,\\
0 & x_C \geq 3;
\end{array}
\right.
$$

$$
\displaystyle \mathcal{Q}_2^4(x_C)=
\pi_C \mathbb{E}[-2D_1-D_2+x_C]^+
=\pi_C 
\left\{ 
\begin{array}{ll}
0 & x_C \leq 0,\\
\displaystyle \frac{x_C^3}{12},& 0 \leq x_C \leq 1,\\
\displaystyle \frac{x_C^2}{4}-\frac{x_C}{4}+ \frac{1}{12},& 1 \leq x_C \leq 2,\\ 
\displaystyle -\frac{x_C^3}{12}+\frac{3}{4}x_C^2-\displaystyle \frac{5}{4}x_C + \frac{3}{4},& 2 \leq x_C \leq 3,\\
x_C - \frac{3}{2} & x_C \geq 3.
\end{array}
\right.
$$

We obtain for $\mathcal{Q}_2$ on $V$
the expression
$$
\mathcal{F}(x_B,x_C)=
2.5x_B^2-\frac{87}{6}x_B+\frac{447}{18}
+ \frac{21}{8}x_C^2-\frac{77}{8}x_C + \frac{273}{24}
$$
which has Hessian matrix 
$$
\mathcal{H}(x_B,x_C) = 
\left( 
\begin{array}{cc}
5 & 0 \\
0 & 21/4
\end{array}
\right)
$$
and 
which is therefore strongly convex with constant of strong
convexity $5=\lambda_{\min}(\mathcal{H}(x_B,x_C))$.
Therefore, problem \eqref{defQ2}
can be written
$$
\begin{array}{l}
\min \; 1.8x_B +0.8x_C +2.5x_B^2-\frac{87}{6}x_B+\frac{447}{18}
+ \frac{21}{8}x_C^2-\frac{77}{8}x_C + \frac{273}{24}\\
1.8 \leq x_B \leq 2.7,\;\;1.35 \leq x_C \leq 1.95
\end{array}
$$
and using a quadratic solver we
obtain the approximate optimal value $12.66$ and approximate solution
$(2.54,1.68)$.
The approximate optimal value is
very close to the approximate optimal values computed by Stochastic Decomposition and SQDP, which gives a validity check for the implementation
of these methods. As expected, we also
check that the optimal value of
\eqref{defQ2} is less than the approximate
optimal values computed by Stochastic Decomposition and SQDP.

\section{Conclusion and extensions}\label{sec:conc}

We  introduced a new method for strongly convex
deterministic optimization problems: QCSC (Quadratic Cuts
for Strongly Convex optimization). We proved the complexity of
this method. 
We extended the idea
of using quadratic approximations in the objective to build models for the recourse functions
of MSPs when the strong convexity assumption holds
for these functions, yielding SQDP algorithm. We also proved
the convergence of SQDP and presented the results of numerical experiments
where SQDP compares favourably with respect to SDDP.

A number of interesting questions arise
from our developments. SQDP applies to the case where recourse functions in MSPs are strongly convex. For linear multistage stochastic programs and discrete uncertainty, this assumption does not hold.
However, for linear two-stage stochastic programs and
continuous distributions, conditions are given in Schultz (1994)
ensuring strong convexity
of the second stage recourse function.
It is natural to study the extension of this result to MSPs: what conditions ensure
strong convexity (and for which  strong convexity parameter) of
recourse functions of linear MSPs having more than 2 stages?
For such problems, another future work will therefore be to extend SQDP
to the case when distributions of $\xi_t$ are continuous. It would also be interesting to apply QCSC to solve
real-life strongly convex
problems 
and SQDP to solve 
other stochastic
two-stage linear (as in Section \ref{two_stage_example}) and multistage strongly convex  optimization problems. \\

\par {\textbf{Competing interests.}} The authors have no competing interests to declare.

\addcontentsline{toc}{section}{References}
\bibliographystyle{plain}
\bibliography{DASC}

\nocite{*}

\newpage

\appendix

\section{Complexity analysis of QCSC}\label{appendix_qcsc_complexity}
This appendix contains the technical lemmas and proofs for the complexity analysis of QCSC in Section~\ref{motivating}. In particular, we establish structural properties of the cutting-plane models \(\Gamma_k\), derive recursive inequalities for the sequence of gaps \(t_k\), provide an upper bound for the initial gap, and conclude with the proof of the complexity theorem.

\vspace*{0.5cm}
 
		The first result below describes some
 basic properties of 
 models $\Gamma_k$.
	
	\begin{lemma}\label{lem:101}
	For every $k \geq 1$, the following statements hold for QCSC:
	    \begin{itemize}
	        \item[a)] for every $\tau \in [0,1]$ we have for every $x \in X$:
	        \begin{equation}
	            \tau \Gamma_{k}(x) + (1-\tau) \left[\ell_f(x,x_k)+\frac{\mu}{2}\|x\|^2\right] \le \Gamma_{k+1}(x) \le \tilde f(x), \label{eq:Gamma_j}
	        \end{equation} 
	        \item[b)] for every $u\in X$, we have
\begin{equation}\label{ineq:Gammaj}
            \Gamma_{k}(u) \ge \Gamma_k(x_k)+ \frac{\mu}{2}\|u-x_k\|^2.
       \end{equation}
	    \end{itemize}
	\end{lemma}
	\begin{proof}
	    a) The relation 
        $\Gamma_{k+1} \leq \tilde f$ on $X$ immediately follows by induction on $k$ and observing that 
        $$
        \ell_{f}(x,x_k)+\frac{\mu}{2}\|x\|^2 \leq f(x) + \frac{\mu}{2}\|x\|^2=\tilde f(x).
        $$
Finally, the inequality
$ \tau \Gamma_{k}(x) + (1-\tau) \left[\ell_f(x,x_k)+\frac{\mu}{2}\|x\|^2\right] \le \Gamma_{k+1}(x)$
immediately follows from the model update relation $\Gamma_{k+1}(\cdot) \leftarrow \max \aa{\Gamma_{k}(\cdot),\ell_{f}(\cdot,x_k) +\frac{\mu}{2}\|\cdot\|^2}$.

	    b) Since model $\Gamma_k$ is a maximum of strongly convex
        functions, it is a strongly convex function. 
        Using the fact that
        $x_{k} \in \argmin_{x \in X} \Gamma_k(x)$, it follows that
        for every $u\in X$, we have 
        $$
        \Gamma_{k}(u) \ge \Gamma_k(x_k)+ \frac{\mu}{2}\|u-x_k\|^2.
        $$
	\end{proof}
	
	\par The following technical result provides an important recursive formula for sequence $\{\Gamma_k(x_k)\}$
	which is used in Lemma \ref{lem:tj} to give
	a recursive formula for 
 sequence
 $\{t_k\}$.

	\begin{lemma}\label{lem:recur}
		Suppose $\tau \in (0,1)$ satisfies \begin{equation}\label{rel:tau1} 
	        \frac{\tau}{1-\tau} \ge \frac{8(M^2 + \bar \varepsilon L)}{\mu\bar \varepsilon}.
	    \end{equation}
  Then models
  $\Gamma_k$ computed in QCSC satisfy  for every 
  $k \geq 1$,
\begin{equation}\label{ineq:mj1}
		    \Gamma_{k+1}(x_{k+1}) \ge \tau \Gamma_{k}(x_{k}) + (1-\tau) \left[ \ell_f(x_{k+1},x_k) + \frac{\mu}{2}\|x_{k+1}\|^2 + \left(\frac{ L}2+\frac{4 M^2}{\bar \varepsilon}\right) \|x_{k+1}-x_k\|^2 \right].
		\end{equation}
	\end{lemma}
    
	\begin{proof}
First, it immediately follows from the assumption over $\tau$ in~\eqref{rel:tau1} that	    \begin{equation}\label{rel:tau2} 
	        \frac{\tau}{1-\tau} \ge \frac{8(M^2 +  \bar \varepsilon L)}{\mu \bar \varepsilon} \ge \frac{1}{\mu} \left(L + \frac{8  M^2}{\bar \varepsilon}\right).
	    \end{equation}
		Using the lower-bound \eqref{eq:Gamma_j} on $\Gamma_{k+1}$, the fact that $\tau<1$, and \eqref{ineq:Gammaj} with $u=x_{k+1}$, we have
		\begin{align*}
		\Gamma_{k+1}(x_{k+1}) 	&\overset{\eqref{eq:Gamma_j}}{\ge} (1-\tau) [\ell_f(x_{k+1},x_k) + \frac{\mu}{2}\|x_{k+1}\|^2] + \tau \Gamma_{k}(x_{k+1}) \\
		& \overset{\eqref{ineq:Gammaj}}{\ge} (1-\tau) [\ell_f(x_{k+1},x_k) + \frac{\mu}{2}\|x_{k+1}\|^2] + \tau \left( \Gamma_k(x_k) + \frac{\mu}{2} \|x_{k+1} -x_k\|^2 \right)\\
        & =  \tau \Gamma_k(x_k) + (1-\tau)  
        \left[ \ell_f(x_{k+1},x_k) + \frac{\mu}{2}\|x_{k+1}\|^2 + \frac{\tau}{1-\tau} \frac{\mu}{2} \|x_{k+1} -x_k\|^2\right]\\
        & \overset{\eqref{rel:tau2}}{\ge} \tau \Gamma_k(x_k) + (1-\tau)  
        \left[ \ell_f(x_{k+1},x_k) + \frac{\mu}{2}\|x_{k+1}\|^2 + \left(\frac{L}2+\frac{4 M^2}{\bar \varepsilon}\right) \|x_{k+1} -x_k\|^2\right]
		\end{align*}
		which ends the proof.
	\end{proof}
\begin{rem} We could improve
the right-hand side of \eqref{ineq:mj1}
replacing $\frac{L}2$ by
$4 L$. We forced the term
$\frac{L}2$ instead in the right-hand side since it will appear naturally
in subsequent computations. 
\end{rem}
 
	The next result,
 which
	plays an important role in the analysis,
 establishes
	a key recursive formula for the sequence $\{t_k\}$ defined in QCSC.
	
	\begin{lemma}\label{lem:tj}
	Suppose $\tau \in (0,1)$ satisfies \eqref{rel:tau1}.     
	    Then sequence $t_k$ computed by QCSC satisfies
\begin{equation}\label{ineq:tj-recur}
         t_{k+1}-\frac{\bar \varepsilon}4 \le \tau \left(t_k -\frac{\bar \varepsilon}4\right).
     \end{equation}
	\end{lemma}
	
	\begin{proof} 
        Using \eqref{ineq:est} with $(u,v)=(x_{k+1},x_k)$ we get 
        \begin{equation} \ell_f(x_{k+1},x_k)  + \frac{ L}2 \|x_{k+1}-x_k\|^2 + 2 M \|x_{k+1}-x_k\| \ge  f(x_{k+1}). \end{equation}
      Combining this relation with
      $\tilde f(x)=f(x)+(\mu/2)\|x\|^2$, we have        \begin{equation}\label{ineq:ellphi}
            \ell_f(x_{k+1},x_k) + \frac{\mu}{2}\|x_{k+1}\|^2  + \frac{ L}2 \|x_{k+1}-x_k\|^2\ge \tilde f(x_{k+1}) - 2 M \|x_{k+1}-x_k\|.
        \end{equation}
        This inequality and \eqref{ineq:mj1} imply that
		\begin{align}
			\Gamma_{k+1}(x_{k+1}) - \tau \Gamma_{k}(x_{k})
			&\overset{\eqref{ineq:mj1}}{\ge} (1-\tau) \left[ \ell_f(x_{k+1},x_k) + \frac{\mu}{2}\|x_{k+1}\|^2 +   \left(\frac{L}2 +\frac{4 M^2}{\bar \varepsilon}\right) \|x_{k+1}-x_k\|^2 \right] \nn \\
			&\overset{\eqref{ineq:ellphi}}{\ge} (1-\tau) \left[ \tilde f(x_{k+1}) - 2  M \|x_{k+1}-x_k\| +  \frac{4 M^2}{\bar \varepsilon} \|x_{k+1}-x_k\|^2 \right] \nn\\
   & = (1-\tau) {\tilde f}(x_{k+1}) + \frac{1-\tau}{\bar \varepsilon}\left(4  M^2\|x_{k+1} -x_k\|^2 - 2 M \bar \varepsilon\|x_{k+1} -x_k\|\right) \nn \\
			&\ge   (1-\tau) {\tilde f}(x_{k+1}) -  \frac{(1-\tau)\bar \varepsilon}{4}, \label{ineq:mj2}
		\end{align}		
		where the last inequality is due to the inequality
		$ a^2-2ab \ge - b^2$ with $a=2 M \|x_{k+1}-x_k\|$ and $b=\bar \varepsilon/2$.
		Using the above inequality and the definitions of $y_{k}$ and $ t_{k}$, respectively, we conclude that
        \begin{align*}
            t_{k+1} & =  \tilde f(y_{k+1}) - \Gamma_{k+1}(x_{k+1}) \overset{\eqref{ineq:mj2}}{\le} \tilde f(y_{k+1}) -\tau \Gamma_k(x_k) - (1-\tau) \tilde f(x_{k+1}) + \frac{(1-\tau)\bar \varepsilon}{4} \\
            & = \tilde f(y_{k+1}) -\tau [\tilde f(y_k)-t_k] - (1-\tau) \tilde f(x_{k+1}) + \frac{(1-\tau)\bar \varepsilon}{4} \\
            & \leq  \tau t_k + \frac{(1-\tau)\bar \varepsilon}{4},
        \end{align*}
        where in the last inequality we have used the fact that
        $$
        \tilde f(y_{k+1}) \leq \tau \tilde f(y_{k}) + (1-\tau)\tilde f(x_{k+1}) 
        $$
        which holds by definition of $y_k$.
        Therefore, the lemma holds.
	\end{proof}
		\vspace*{0.5cm}
	The next lemma gives a uniform bound on $t_{1}$.
    \begin{lemma}\label{lem:t1} 
    Assume that $X$
    is bounded with diameter at most 
    $D$. Then we have $ t_{1}\le \bar t(D)$
		where
        $\bar t(D)$ is given by
        \eqref{def:bar t}.
	\end{lemma}
	
	\begin{proof}
	    Using relation 
       $\Gamma_{1}(\cdot)= \ell_f(\cdot,x_{0})+\frac{\mu}{2}\|\cdot\|^2$, we have
	    \begin{align*}
	        t_{1} &=\tilde f(y_{1})- \Gamma_1(x_1)
         \leq 
f(x_{1}) + \frac{\mu}{2}\|x_1\|^2 
         -\ell_f(x_1,x_{0})-\frac{\mu}{2}\|x_1\|^2
	        = f(x_{1}) - \ell_f(x_{1},x_{0}) \\
	        & \le  2 M \|x_{1}-x_{0}\| + \frac{L}{2} \|x_{1}-x_{0}\|^2
	        \le M^2 + \left( \frac{L}{2}+1\right)\|x_{1}-x_{0}\|^2 \leq M^2 + 
       \left( \frac{{L}}{2} +1\right) D^2
	    \end{align*}
		where the second inequality is due to \eqref{ineq:est} with $(u,v)=(x_{1}, x_{0})$, and the third inequality is due to the fact that $2ab \le a^2+b^2$ for every $a,b \in \R$.
	\end{proof}

\vspace*{0.2cm}

We are now ready to prove Theorem \ref{complcssc}.\\

\par {\textbf{Proof of Theorem \ref{complcssc}.}}
    Define $\tau \in (0,1)$ by
\begin{equation}\label{partchtau}
\frac{1}{\tau}=1+\frac{\mu \bar \varepsilon}{8(M^2+\bar \varepsilon L)}. 
\end{equation}
Observe that for this choice of $\tau$, relation
$$
  \frac{\tau}{1-\tau} \ge \frac{8(M^2 + \bar \varepsilon L)}{\mu\bar \varepsilon}
$$
holds and therefore Lemmas
\ref{lem:recur} and
\ref{lem:tj} can be applied. In particular,
relation \eqref{ineq:tj-recur} holds for this value of $\tau$.

Using Lemma \ref{lem:tj}
and the relation $\tau \leq e^{\tau-1}$, we get
for every $k \geq 1$:
\begin{align}
    t_k  & \leq  \frac{\bar \varepsilon}{4} +\tau^{k-1}(t_1-\frac{\bar \varepsilon}{4})
    & \text{by induction on \eqref{ineq:tj-recur}} \nonumber\\
 & \leq \frac{\bar \varepsilon}{4} +\tau^{k-1} \bar t(D) & \text{by Lemma \ref{lem:t1}} \nonumber\\
 & \leq \frac{\bar \varepsilon}{4}
 +e^{(\tau-1)(k-1)} \bar t(D) & \text{since } \tau \leq e^{\tau-1}.\label{tjfinal}
\end{align}
Further, 
$$ \tau-1 = \frac{- \frac{\mu \bar \varepsilon}{8(M^2+\bar \varepsilon L)}}{1+\frac{\mu \bar \varepsilon}{8(M^2+\bar \varepsilon L)}}
= - \frac{1}{1+\frac{8(M^2+\bar \varepsilon L)}{\mu \bar \varepsilon}}.$$
So for $k \geq 1+\left(1 + \frac{8(M^2+\bar \varepsilon L)}{\mu \bar \varepsilon}\right)\log(4\bar t(D) / (3 \bar \varepsilon))$, we have 
$
(k-1)(\tau-1) \leq \log(3 \bar \varepsilon/(4 \bar t(D)))
$ or equivalently
$\exp((\tau-1)(k-1)) \leq 3 \bar \varepsilon/(4\bar t(D))$, which plugged into
\eqref{tjfinal} gives
$t_k \leq \frac{\bar \varepsilon}{4} + \frac{3\bar \varepsilon}{4} = \bar \varepsilon$.

Finally, for $k \geq 1+\left(1 + \frac{8(M^2+\bar \varepsilon L)}{\mu \bar \varepsilon}\right)\log(4\bar t(D) / (3 \bar \varepsilon))$, we also have
$$
0 \leq \tilde f(y_k)-\tilde f_* \leq 
\tilde f(y_k)-\Gamma_k(x_k) =t_k \leq \bar \varepsilon
$$
and $y_k \in X$ is an $\bar \varepsilon$-optimal solution.$\hfill \smallblacksquare$

\end{document}